\documentclass[12pt]{amsart}
\pdfoutput=1

\usepackage[utf8]{inputenc}
\usepackage[T1]{fontenc}
\usepackage{latexsym}
\usepackage{amsxtra}
\usepackage{amsmath}
\usepackage{amsfonts}
\usepackage{amssymb}
\usepackage{blkarray}
\usepackage{multirow}
\usepackage{pgf,tikz}
\usepackage{mathrsfs}
\usepackage{mathdots}
\usepackage{pdfpages}
\usepackage[colorlinks,linkcolor=blue,anchorcolor=black,citecolor=blue,filecolor=blue,menucolor=blue,runcolor=blue,urlcolor=blue]{hyperref}
\usetikzlibrary{arrows}

\tikzstyle{vertex}=[circle, draw, inner sep=0pt, minimum size=6pt]
\newcommand{\vertex}{\node[vertex]}

\newtheorem{theorem}{Theorem}[section]
\newtheorem{lemma}[theorem]{Lemma}
\newtheorem{corollary}[theorem]{Corollary}

\newtheorem{conjecture}[theorem]{Conjecture}
\newtheorem{hypothesis}[theorem]{Hypothesis}
\newtheorem{claim}{Claim}[theorem]
\newenvironment{subproof}[1][\proofname]{%
  \begin{proof}[#1]%
}{%
  \end{proof}%
}
\numberwithin{equation}{section}

\theoremstyle{definition}
\newtheorem{definition}[theorem]{Definition}

\newtheorem{notation}[theorem]{Notation}

\theoremstyle{remark}
\newtheorem{remark}[theorem]{Remark}

\newcommand{\bFp}{\mathbb{F}_p}
\newcommand{\rtp}{rank-($\leq t$) perturbation }
\newcommand{\Or}{Or(G,R)}
\newcommand{\Np}{N^{+}}

\DeclareMathOperator{\pert}{pert}
\DeclareMathOperator{\dist}{dist}

\begin{document}
\title[Perturbed dyadic matroids]{On perturbations of highly connected dyadic matroids}

\author{Kevin Grace}
\email{kgrace3@lsu.edu}
\author{Stefan H. M. van Zwam}
\email{svanzwam@math.lsu.edu}
 \address{Department of Mathematics\\
 Louisiana State University\\
 Baton Rouge, Louisiana}

\thanks{The research for this paper was supported by National Science Foundation grant 1500343.}

\date{\today}

\begin{abstract}
Geelen, Gerards, and Whittle \cite{ggw15} announced the following result: let $q = p^k$ be a prime power, and let $\mathcal{M}$ be a proper minor-closed class of $\mathrm{GF}(q)$-representable matroids, which does not contain $\mathrm{PG}(r-1,p)$ for sufficiently high $r$. There exist integers $k, t$ such that every vertically $k$-connected matroid in $\mathcal{M}$ is a rank-$(\leq t)$ perturbation of a frame matroid or the dual of a frame matroid over $\mathrm{GF}(q)$. They further announced a characterization of the perturbations through the introduction of subfield templates and frame templates.

We show a family of dyadic matroids that form a counterexample to this result. We offer several weaker conjectures to replace the ones in \cite{ggw15}, discuss consequences for some published papers, and discuss the impact of these new conjectures on the structure of frame templates.
\end{abstract}

\maketitle
\section{Introduction}
\label{Introduction}
Robertson and Seymour profoundly transformed graph theory with their Graph Minors Theorem \cite{rs04}. Geelen, Gerards, and Whittle are on track to do the same for matroid theory with their Matroid Structure Theory for matroids representable over a finite field (see, e.g. \cite{ggw07}). The theorem they intend to prove is the following:

\begin{conjecture}[Matroid Structure Theorem, rough idea]\label{con:matroidstructure}
  Let $\mathbb{F}$ be a finite field, and let $\mathcal{M}$ be a proper minor-closed class of $\mathbb{F}$-representable matroids. If $M \in \mathcal{M}$ is sufficiently large and has sufficiently high branch-width, then $M$ has a tree-decomposition, the parts of which correspond to mild modifications of matroids representable over a proper subfield of $\mathbb{F}$, or to mild modifications of frame matroids and their duals.
\end{conjecture}
The words ``tree-decomposition'', ``parts'', ``correspond to'', and ``mild modifications'' need (a lot of) elaboration, and hide over 15 years of very hard work. Whittle \cite{WhiQuote} described the proof of Rota's Conjecture, which has the Matroid Structure Theorem as a major ingredient, as follows:

\begin{quotation}
    ``It's a little bit like discovering a new mountain -- we've crossed many hurdles to reach a new destination and we have returned scratched, bloodied and bruised from the arduous journey -- we now need to create a pathway so others can reach it.''
\end{quotation}

In this paper we will only focus on the last part of Conjecture \ref{con:matroidstructure}. Geelen, Gerards, and Whittle announced without proof a theorem about that part \cite[Theorem 3.1]{ggw15} that we will repeat here as Conjecture \ref{con:perturb}. First, we require some definitions. An $\mathbb{F}$-\emph{represented matroid} (or simply \emph{represented matroid} if the field is understood from the context) is a matroid with a fixed class of representation matrices over $\mathbb{F}$ that are row equivalent up to column scaling and removal of zero rows. A \emph{represented frame matroid} is a matroid with a representation matrix $A$ that has at most two nonzero entries per column. The matroids we will be working with in Section \ref{sec:the_construction} are dyadic and therefore ternary. Since ternary matroids are uniquely $\mathrm{GF}(3)$-representable \cite{bl76}, we will not make any distinction between matroids and represented matroids in that case. We also extend this convention to binary matroids, particularly complete graphic matroids, since every binary matroid that is representable over some field $\mathbb{F}$ is uniquely $\mathbb{F}$-representable.

A matroid (or represented matroid) is \emph{vertically $k$}-connected if, for every separation $(A,B)$ of order less than $k$, one of $A$ and $B$ spans $E(M)$. If a matroid (or represented matroid) $M$ is vertically $k$-connected, then $M^*$ is \emph{cyclically $k$}-connected. A \emph{rank-$(\leq t)$ perturbation} of a represented matroid $M$ is the represented matroid obtained by adding a matrix of rank at most $t$ to the representation matrix of $M$.

\begin{conjecture}[{\cite[{Theorem 3.1}]{ggw15}}]\label{con:perturb}
Let $\mathbb F$ be a finite field and let $\mathcal M$ be a proper minor-closed class of $\mathbb F$-represented matroids. Then there exist $k,t\in\mathbb Z_+$ such that each vertically $k$-connected member of $\mathcal M$ is a rank-$(\leq t)$ perturbation of an $\mathbb F$-represented matroid $N$, such that either
\begin{itemize}
\item[(i)] $N$ is a represented frame matroid, 
\item[(ii)] $N^*$ is a represented frame matroid, or
\item[(iii)] $N$ is confined to a proper subfield of $\mathbb F$.
\end{itemize}
\end{conjecture}

In this paper we present a counterexample to this conjecture. In particular, we build a family of dyadic matroids that are vertically $k$-connected for any desired $k$, and not a bounded-rank perturbation of either a represented frame matroid or the dual of a represented frame matroid. The construction starts with a cyclically $k$-connected graph $G$, modifies it at a number of vertices that grows with $|V(G)|$, and dualizes the resulting matroid. We detail the construction, and prove its key properties, in Section \ref{sec:the_construction}.

Vertical connectivity and cographic matroids are not very compatible notions. Because of this, our examples, which are very sparse, arise only in situations where one might expect the second outcome of the conjecture to hold. For this reason, the forthcoming proof of the Matroid Structure Theorem itself is not jeopardized, and versions of Conjecture \ref{con:perturb} can be recovered.  In Section \ref{sec:discussion}, we provide several such conjectures. Included in Section \ref{sec:discussion} is Section \ref{sub:consequences}, where we discuss consequences for \cite{gvz17} and \cite{nvz15}. In Section \ref{sec:refined_templates}, we discuss consequences for the notion of \emph{frame templates}, introduced in \cite{ggw15} to describe the perturbations in more detail.

\section{Preliminaries}
\label{Preliminaries}
Unexplained notation and terminology will generally follow Oxley ~\cite{o11}. One exception is that we denote the vector matroid of a matrix $A$ by $M(A)$, rather than $M[A]$. The following characterization of the dyadic matroids was shown by Whittle in \cite{w97}.

\begin{theorem}
\label{thm:dyadic}
A matroid is \emph{dyadic} if and only if it is representable over both $\mathrm{GF}(3)$ and $\mathrm{GF}(5)$.
\end{theorem}

We will need some definitions and results related to bounded-rank perturbations of represented matroids. The next three definitions are from \cite{ggw15}.

\begin{definition}
Let $M_1$ and $M_2$ be $\mathbb{F}$-represented matroids on a common ground set. Then $M_2$ is a \emph{rank-$(\leq t)$ perturbation} of $M_1$ if there exist matrices $A_1$ and $P$ such that $M(A_1) = M_1$, the rank of $P$ is at most $t$, and $M(A_1 + P) = M_2$.
\end{definition}

\begin{definition}
 Let $M_1$ and $M_2$ be $\mathbb{F}$-represented matroids on ground set $E$. If there is some $\mathbb{F}$-represented matroid $M$ on ground set $E\cup\{e\}$ such that $M_1=M\backslash e$ and $M_2=M/e$, then $M_1$ is an \textit{elementary lift} of $M_2$, and $M_2$ is an \textit{elementary projection} of $M_1$.
\end{definition}

Note that an elementary lift of a represented matroid $M(A)$ can be obtained by appending a row to $A$.

\begin{definition}
Let $M_1$ and $M_2$ be $\mathbb{F}$-represented matroids on a common ground set. We denote by $\dist(M_1,M_2)$ the minimum number of elementary lifts and elementary projections needed to transform $M_1$ into $M_2$, and we denote by $\pert(M_1,M_2)$ the smallest integer $t$ such that $M_2$ is a rank-($\leq t$) perturbation of $M_1$.
\end{definition}

The following observation will be quite useful; in particular, we use it to prove Lemma \ref{pert&dist} below.

\begin{remark}
\label{contract}
Suppose that $M_1=M(A_1)$ is a rank-$(\leq t)$ perturbation of $M_2=M(A_2)$. Let $P$ be the matrix of rank at most $t$ such that $A_1+P=A_2$. Let $\{v_1,v_2,\ldots,v_t\}$ be a basis for the row space of $P$. Note that neither $A_1$, $P$, nor $A_1+P$ need have full row rank. If $r=r(M_1)$, then we may assume that $P$ has $r+t$ rows. If $a_{i,j}\in\mathbb{F}$ for all $i,j$, let $a_{i,1}v_1+a_{i,2}v_2+\ldots+a_{i,t}v_t$ be the $i$-th row of $P$. Then $M_1$ can be obtained by contracting $C$ from the represented matroid obtained from the following matrix.
\begin{center}
\begin{tabular}{|ccccccc|ccc|}
\multicolumn{7}{c}{}&\multicolumn{3}{c}{$C$}\\
\hline
&&&$v_1$&&&&&&\\
&&&$\vdots$&&&&&$-I$&\\
&&&$v_t$&&&&&&\\
\hline
&&&&&&&$a_{1,1}$&$\dots$&$a_{1,t}$\\
&&&$A_1$&&&&$\vdots$&$\ddots$&$\vdots$\\
&&&&&&&$a_{r+t,1}$&$\dots$&$a_{r+t,t}$\\
\hline
\end{tabular}
\end{center}
\end{remark}

Lemma \ref{pert&dist} appears in ~\cite{ggw15} as Lemma 2.1; however, no proof was given in ~\cite{ggw15}. We will need it to prove our main result, so we give a proof here.

\begin{lemma}[{\cite[Lemma 2.1]{ggw15}}]
\label{pert&dist}
If $M_1$ and $M_2$ are $\mathbb{F}$-represented matroids on the same ground set, then
\[ \pert(M_1,M_2)\leq \dist(M_1,M_2)\leq 2\pert(M_1,M_2).\]
\end{lemma}

\begin{proof}
 A rank-$(\leq t)$ perturbation of a represented matroid $M_1$ can be obtained by successively adding $t$ rank-$1$ matrices to some matrix $A_1$ with $M(A_1)=M_1$. Therefore, we can prove this result inductively by considering the behavior of elementary lifts, elementary projections, and rank-$1$ perturbations. An elementary lift of $M_1$ can be obtained by adding the rank-$1$ matrix $\left[\begin{array}{c}
v\\
\hline
0
\end{array}
\right]$, for some vector $v$, to the matrix $\left[\begin{array}{ccc}
0&\cdots&0\\
\hline
&A_1&
\end{array}
\right]$, which represents $M_1$. Thus, every elementary lift of a represented matroid is also a rank-$1$ perturbation of the represented matroid. Now, since $M_1$ is a rank-($\leq t$) perturbation of $M_2$ if and only if $M_2$ is a rank-($\leq t$) perturbation of $M_1$ and since $M_1$ is an elementary lift of $M_2$ if and only if $M_2$ is an elementary projection of $M_1$, we also have that every elementary projection of a represented matroid is a rank-$1$ perturbation of the represented matroid. The converse of these statements is not true in general; however, we will show that every rank-$1$ perturbation of a represented matroid can be obtained by performing an elementary lift followed by an elementary projection.

Suppose that $M_2$ is a rank-$1$ perturbation of $M_1$. By Remark \ref{contract}, there are vectors $v$ and $w$ and a matrix $A_1$ with $M(A_1)=M_1$ such that $M_2$ is obtained from the matrix $A=\left[
\begin{array}{ccc|c}
&v& & -1 \\
\hline
&A_1& & w\\
\end{array}
\right]$ by contracting the element represented by the last column. The represented matroid obtained from $A'=\left[\begin{array}{ccc}
&v&\\
\hline
&A_1&
\end{array}
\right]$ is an elementary lift of $M_1$. Since $M_2$ is obtained from $M(A)$ by contracting the element represented by the last column, $M_2$ is an elementary projection of $M(A')$.

The fact that a rank-$1$ perturbation can be obtained by at least one elementary lift or projection implies that $\pert(M_1,M_2)\leq\dist(M_1,M_2)$. The fact that a rank-$1$ perturbation can be obtained by at most two elementary lifts and projections implies that $\dist(M_1,M_2)\leq2\pert(M_1,M_2)$.
\end{proof}

In order to prove our main result, we will need some lemmas regarding duality. The first lemma is an easy corollary of Lemma \ref{pert&dist}. In fact, the following lemma still holds if $2t$ is replaced by $t$, but that best possible result is not necessary for our purposes.

\begin{lemma}
 \label{dualperturbations}
Suppose that $M_2$ is a rank-$(\leq t)$ perturbation of $M_1$. Then $M^*_2$ is a rank-$(\leq 2t)$ perturbation of $M^*_1$.
\end{lemma}

\begin{proof}
By Lemma \ref{pert&dist} and duality of elementary lifts and elementary projections, we have
 \[\pert(M_1^*,M_2^*)\leq\dist(M_1^*,M_2^*)=\dist(M_1,M_2)\leq2\pert(M_1,M_2)\leq 2t.\]
\end{proof}

Let $\varepsilon(M)=|si(M)|$; that is, $\varepsilon(M)$ is the number of rank-$1$ flats of $M$. The next lemma, proved by Nelson and Walsh ~\cite{nw17}, gives a bound on $\varepsilon(M)$, when $M$ is the dual of a frame matroid. We use this and Lemma \ref{sizedifference} to prove Lemma \ref{graphicvsframedual} below.

\begin{lemma}[{\cite[Lemma 6.2]{nw17}}]
 \label{framedualbound}
If $M^*$ is a frame matroid, then $\varepsilon(M)\leq 3r(M)$.
\end{lemma}

Although \cite{nw17} was based on \cite{ggw15}, the previous lemma was proved independently of \cite{ggw15} and remains accurate for that reason.

\begin{lemma}
 \label{sizedifference}
If $M$ is a rank-($\leq t$) perturbation of a $\mathrm{GF}(q)$-represented matroid $N$, then $\varepsilon(M)\leq q^t\varepsilon(N)+\sum_{i=0}^{t-1} q^i$.
\end{lemma}

\begin{proof}
 We proceed by induction on $t$. If $t=0$, then $M=N$, and the result is clear. Now suppose the result holds for rank-($\leq t'$) perturbations for all $t'<t$. Since $M$ is a \rtp
of $N$ there is some represented matroid $M'$ such that $M'$ is a rank-($\leq t-1$) perturbation of $N$ and $M$ is a rank-($\leq 1$) perturbation of $M'$. Thus, there are matrices $A$ and $P$ such that $M'=M(A)$, the rank $P$ is $1$, and $M=M(A+P)$. We will show that the nonloop elements in a rank-$1$ flat of $M'$ become members of at most $q$ distinct rank-$1$ flats in $M$. Let $\{a_1,a_2,\ldots,a_{q-1},a_q\}$ be the elements of $\mathrm{GF}(q)$, with $a_q=0$, and let $v$ be a nonzero column in $A$ indexed by an element in a rank-$1$ flat $F$ of $M'$. Then the nonloop elements of $F$ are each represented by a column $a_iv$ for some $i$ such that $1\leq i\leq q-1$. Similarly, let $w$ be a nonzero column of $P$. Then every column of $P$ is represented by a column $a_iw$ for some $i$ such that $1\leq i\leq q$. Thus, every element in $F$ that is not a loop in $M'$ will be represented in $A+P$ by a column of the form $a_iv+a_jw=a_i(v+a_i^{-1}a_jw)$, where $1\leq i\leq q-1$ and $1\leq j\leq q$. There are $q$ distinct possible values for $a_i^{-1}a_j$; therefore, the elements of $F$ that are not loops in $M'$ are in at most $q$ distinct rank-$1$ flats in $M$. Moreover, after $P$ is added to $A$, loops in $M'$ will be represented by columns of the form $a_jw$. This accounts for one additional rank-$1$ flat in $M$. Thus, $\varepsilon(M)\leq q\varepsilon(M')+1$. By the induction hypothesis, we have $\varepsilon(M)\leq q(q^{t-1}\varepsilon(N)+\sum_{i=0}^{t-2} q^i)+1=q^t\varepsilon(N)+\sum_{i=0}^{t-1} q^i$, which proves the result.
\end{proof}

\begin{lemma}
\label{graphicvsframedual}
 Let $t$ be a positive integer, and let $\mathbb{F}=\mathrm{GF}(q)$. Then there are finitely many integers $r$ such that the complete graphic matroid $M(K_{r+1})$ is a rank-($\leq t$) perturbation of the dual of an $\mathbb{F}$-represented frame matroid.
\end{lemma}

\begin{proof}
Suppose $M$ is a \rtp of an $\mathbb{F}$-represented matroid $N$, and let $r=r(M)$. Combining the previous two lemmas, we have $\varepsilon(M)\leq q^t(3r(N))+\sum_{i=0}^{t-1} q^i$. Since $M$ is a \rtp of $N$, we have $r(M)\leq r(N)+t$. Therefore, $\varepsilon(M)\leq 3q^t(r-t)+\sum_{i=0}^{t-1} q^i$. Since $q$ and $t$ are constant, this expression is less than $\binom{r+1}{2}=\varepsilon(M(K_{r+1}))$ for all sufficiently large $r$.
\end{proof}

\begin{notation}
\label{function}
 Let $g(q,t)$ be the least value $n$ such that for all $n'\geq n$, the complete graphic matroid $M(K_{n'})$ is not represented by any represented matroid that is a rank-($\leq t$) perturbation of the dual of a represented frame matroid over $\mathrm{GF}(q)$.
\end{notation}

Lemma \ref{graphicvsframedual} can be restated as saying that $g(q,t)$ is finite for every prime power $q$ and positive integer $t$.

As stated in the introduction, our construction makes use of highly cyclically connected graphs. The next two results allow us to specify some additional details about these graphs. The following lemma seems to be fairly well known; however, we were unable to find an explicit proof in the literature. For the sake of completeness, we state the result and provide a proof, obtained by combining some older results.

\begin{lemma}
 \label{cubic}
For every positive integer $k$, there is a cyclically $k$-connected cubic graph.
\end{lemma}

\begin{proof}
 There is a cubic Cayley graph of girth $g\geq k$. (See, for example, Biggs ~\cite{b89} or Jajcay and \v Sir\'a\v n ~\cite[Theorem 2.1]{js11}. In particular, ~\cite{js11} contains a nice summary of various related results.) Since Cayley graphs are vertex-transitive, a result of Nedela and \v Skoviera ~\cite[Theorem 17]{ns95} states that such a graph has cyclic connectivity $g$.
\end{proof}

Thomassen ~\cite[Corollary 3.2]{t83} showed the following.

\begin{theorem}
\label{thomassen}
 There is a function $\xi$ such that a graph $G$ with minimum degree at least $3$ and girth at least $\xi(n)$ has a minor isomorphic to $K_n$.
\end{theorem}

Finally, we clarify some notation and terminology, following ~\cite{nw17}. These will be used in Sections \ref{sec:discussion} and \ref{sec:refined_templates}. If $\mathbb{F}$ is a field and $A\cap B=\emptyset$, then we identify the vector space $\mathbb{F}^A\times\mathbb{F}^B$ with $\mathbb{F}^{A\cup B}$. If $U\subseteq\mathbb{F}^E$ and $X\subseteq E$, then $U[X]$ is the set of vectors consisting of all coordinate projections of $u$ onto $X$ for $u\in U$. If $\Gamma\subseteq\mathbb{F}$, then $\Gamma U=\{\gamma u|\gamma\in\Gamma, u\in U\}$. If $U$ and $W$ are additive subgroups of $\mathbb{F}^E$, then $U$ and $W$ are \emph{skew} if $U\cap W=\{0\}$.

\section{The Construction}
\label{sec:the_construction}

Our construction involves repeated use of the generalized parallel connection of a matroid with copies of $M(K_5)$ over a copy of $M(K_4)$ represented in a specific way. The next two results specify that representation. Both results are easily checked, so we state them without proof.

\begin{lemma}
\label{K5matrix}
 The following matrix represents $M(K_5)$ over all fields of characteristic other than $2$:

\[A=\begin{bmatrix}
1&1&0&0&1&1&0&1&0&1\\
-1&1&1&1&0&0&0&1&1&0\\
0&0&-1&1&-1&1&0&0&1&1\\
0&0&0&0&0&0&1&1&1&1\\
\end{bmatrix}.\]
\end{lemma}

\begin{lemma}
\label{lem:K4signedgraph}
 The signed graph shown in Figure \ref{M(K_4)}, 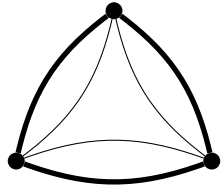
\begin{figure}
\[\begin{tikzpicture}[x=1.3cm, y=1cm,
    every edge/.style={
        draw
        }
]
	\vertex[fill] (1) at (0,0) {};
	\vertex[fill] (2) at (2,0) {};
	\vertex[fill] (3) at (1,2) {};
	\path
		(1) edge[bend right=20, line width=2.0pt] (2)
		(2) edge[bend right=20] (1)
		(3) edge[bend right=20] (2)
		(2) edge[bend right=20, line width=2.0pt] (3)
		(1) edge[bend right=20] (3)
		(3) edge[bend right=20, line width=2.0pt] (1)
	;
\end{tikzpicture}\]
\caption{A signed-graphic representation of $M(K_4)$}
  \label{M(K_4)}
\end{figure}with negative edges printed in bold, represents $M(K_4)$.

\end{lemma}
This representation of $M(K_4)$ has been encountered before, for example in \cite{z90,g90,sq07}.

\begin{definition}
 \label{def:ornamentation}
Let $G$ be a cubic graph, and $R\subseteq V(G)$. For each vertex $v_i$ of $R$, perform the operation of altering the graph on the left in Figure \ref{GtoG'} \begin{figure}
\[\begin{tikzpicture}[x=.8cm, y=.5cm,
    every edge/.style={
        draw
        }
]
	\vertex[fill] (v_i) at (1,2) [label=left:$v_i$]{};
	\vertex[fill] (1) at (0,0) [label=left:$x$]{};
	\vertex[fill] (2) at (2,0) [label=right:$y$]{};
	\vertex[fill] (3) at (1,4) [label=above:$z$]{};
	\vertex[fill] (4) at (5,0) [label=left:$x$]{};
	\vertex[fill] (5) at (9,0) [label=right:$y$]{};
	\vertex[fill] (6) at (7,5) [label=above:$z$]{};
	\vertex[fill] (7) at (6,1) {};
	\vertex[fill] (8) at (8,1) {};
	\vertex[fill] (9) at (7,3) {};
	\path
		(7) edge[bend right=20, line width=2.0pt] (8)
		(8) edge[bend right=20] (7)
		(9) edge[bend right=20] (8)
		(8) edge[bend right=20, line width=2.0pt] (9)
		(7) edge[bend right=20] (9)
		(9) edge[bend right=20, line width=2.0pt] (7)
		(1) edge (v_i)
		(2) edge (v_i)
		(3) edge (v_i)
		(4) edge (7)
		(5) edge (8)
		(6) edge (9)
	;
\end{tikzpicture}\]
\caption{Changing $G$ to $G'$}
  \label{GtoG'}
\end{figure}
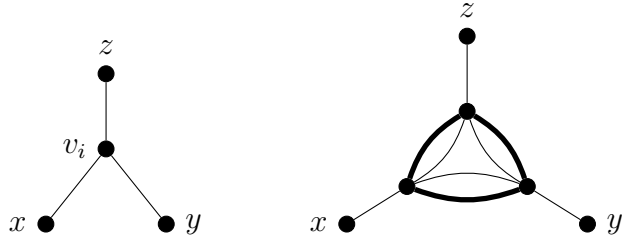to become the signed graph on the right, with negative edges printed in bold. Let $G'$ be the signed graph that results from performing this operation on every vertex in $R$. Note that $G'$ contains $|R|$ copies of the signed-graphic representation of $M(K_4)$ described in Lemma \ref{lem:K4signedgraph}. Let $X_1,X_2,\dots, X_{|R|}$ be the edge sets of these representations of $M(K_4)$. For each $X_i$, take the generalized parallel connection\footnote{Each $X_i$ is a modular flat of $M(K_5)$, which is uniquely representable over any field. Therefore, these generalized parallel connections are well defined.} of $M(G')$ with a copy of $M(K_5)$ over $X_i$. Delete the $X_i$, and call the resulting matroid the \emph{ornamentation} of $(G,R)$, denoted by $Or(G,R)$.
\end{definition}

\begin{lemma}
 \label{lem:dyadic}
For any cubic graph $G$, with $R\subseteq V(G)$, the ornamentation $Or(G,R)$ is dyadic and has $M(G)$ as a minor.
\end{lemma}

\begin{proof}
It is well-known that signed-graphic matroids are dyadic (see, for example, \cite[Lemma 8A.3]{z82}). The construction of $\Or$ involves generalized parallel connections of a signed graph with copies of $M(K_5)$ over a common representation of $M(K_4)$. Thus, a result of Mayhew, Whittle, and Van Zwam \cite[Theorem 3.1]{mwvz11} implies that $\Or$ is dyadic.

Note that $Or(G,R)$ is the result of taking $|R|$ copies of the submatrix $[\frac{1\hspace{12pt} 1 \hspace{12pt} 1}{-I}]$ of the signed incidence matrix of $G$, with columns indexed by the set $\{1_i,2_i,3_i\}$, and altering each of them to become a copy of the following matrix:
\[\begin{blockarray}{ccccccc}
1_i & 2_i & 3_i & d_i & e_i & f_i & g_i\\
\begin{block}{[ccccccc]}
1&0&0&0&1&0&1\\
0&1&0&0&1&1&0\\
0&0&1&0&0&1&1\\
-1&0&0&0&0&0&0\\
0&-1&0&0&0&0&0\\
0&0&-1&0&0&0&0\\
0&0&0&1&1&1&1\\
\end{block}
\end{blockarray}.\] One can check that in the vector matroid of the above submatrix, deleting $g_i$ and contracting $\{d_i,e_i,f_i\}$ results in the vector matroid of $[\frac{1\hspace{12pt} 1 \hspace{12pt} 1}{-I}]$. Therefore, $M(G)$ is a minor of $Or(G,R)$.
\end{proof}

\begin{definition}
 \label{def:gadget}
In the matrix given in the proof of the previous lemma, let $F_i=\{d_i,e_i,f_i,g_i\}$. We will call each $F_i$ a \textit{gadget}.
\end{definition}

We are now ready to prove our main result.

\begin{theorem}
\label{thm:construction}
For every $k,t\in\mathbb{Z}_+$, there exists a vertically $k$-connected dyadic matroid that is not a rank-$(\leq t)$ perturbation of either a frame matroid or the dual of a frame matroid.
\end{theorem}

\begin{proof}
Let $g$ and $\xi$ be the functions given in Notation \ref{function} and Theorem \ref{thomassen}, respectively. We must define several constants that will be used throughout this proof. First, let $d>3(2^{\lfloor\log_2(k)\rfloor+1}-1)$, and let $c\geq 2t+20(3^{16t})$. We define \[m=\max\{c(3^d+3)+7,g(3,2t)\}.\]
Finally, let \[h=\max\{k+1,\xi(m)\}.\]

Let $G$ be a cyclically $h$-connected cubic graph. Such a graph exists by Lemma \ref{cubic}. This implies that $G$ has girth at least $h\geq\xi(m)$. By Theorem \ref{thomassen}, $G$ has a minor $H$ isomorphic to $K_m$. Let $C$ and $D$ be the sets of edges such that $G/C\backslash D=H$. Each vertex of $H$ is obtained by contracting all edges  in a subtree of $G\backslash D$. Thus, there is a function $\phi:V(G)\rightarrow V(H)$ such that $\phi(w)=v$ if $w$ is a vertex in the subtree of $G\backslash D$ that is contracted to result in $v$.

\begin{claim}
\label{cla:R}
There is a set $R=\{v_1,v_2,\ldots,v_c\}\subseteq V(G)$ of size $c$ and $c+1$ pairwise disjoint sets $\{a_0,b_0,c_0\}$,$\{a_1,b_1,c_1\}$,$\ldots$,$\{a_c,b_c,c_c\}\subseteq V(H)$ that are also disjoint from $\phi(R)$ such that
\begin{enumerate}
 \item the members of $R$ are pairwise a distance of at least $d$ from each other in $G$, and
\item for each integer $i$ with $1\leq i\leq c$, if $T_i$ is the subtree of $G\backslash D$ that is contracted to obtain $\phi(v_i)$, then there is a set of vertices $\{a'_i,b'_i,c'_i\}\subseteq V(T_i)$ (possibly some or all of $a'_i$, $b'_i$, and $c'_i$ are equal to $v_i$) and three internally disjoint subpaths of $T_i$ from $v_i$ to $a'_i$, $b'_i$, and $c'_i$ such that $a'_i$, $b'_i$, and $c'_i$ are neighbors in $G\backslash D$ of some vertex in $\phi^{-1}(a_i)$, $\phi^{-1}(b_i)$, and $\phi^{-1}(c_i)$, respectively.
\end{enumerate}
\end{claim}

\begin{subproof}
Suppose $\{v_1, a_1, b_1, c_1\}, \ldots, \{v_{k-1}, a_{k-1}, b_{k-1}, c_{k-1}\}$ and $\{a_0,b_0,c_0\}$ were chosen to satisfy (1) and (2), with $k$ maximal. Also suppose, for a contradiction, that $k-1<c$. Since $G$ is cubic, there are at most $\sum_{i=0}^{d-1}3^i<3^d$ vertices in $G$ whose distance from some $v_i$ is less than $d$. Thus, after choosing $\{v_1,\ldots,v_{k-1}\}\subseteq V(G)$ and $\{a_0,b_0,c_0\}$,$\{a_1, b_1, c_1\}$, $\ldots$, $\{a_{k-1}, b_{k-1}, c_{k-1}\}\subseteq V(H)$, there are at least $m-(k-1)(3^d+3)-3>m-c(3^d+3)-3\geq4$ vertices $w$ in $V(H)-(\{a_0,b_0,c_0\}\cup\{\phi(v_1),a_1,b_1,c_1\}\cup\ldots\cup\{\phi(v_1),a_1,b_1,c_1\})$ such that every vertex in $\phi^{-1}(w)$ is at a distance of at least $d$ from each member of $\{v_1,\ldots,v_{k-1}\}$. (In the expression $m-(k-1)(3^d+3)-3$, the $+3$ comes from the sets $\{a_i,b_i,c_i\}$ for $i>0$, and the $-3$ comes from $\{a_0,b_0,c_0\}$.) Choose one of these vertices $w$ to be $\phi(v_k)$, and let three of the others be $\{a_k,b_k,c_k\}$.

Since each of $a_k$, $b_k$, and $c_k$ is a neighbor of $\phi(v_k)$ in $H$, there must be vertices $\{a'_k,b'_k,c'_k\}\subseteq V(T_k)$ that are neighbors in $G\backslash D$ of some vertex in $\phi^{-1}(a_k)$, $\phi^{-1}(b_k)$, and $\phi^{-1}(c_k)$, respectively. If $a'_k=b'_k=c'_k$, then let $v_k=a'_k=b'_k=c'_k$. If two of $\{a'_k,b'_k,c'_k\}$ are equal, say $a'_k=b'_k$, then let $v_k=a'_k=b'_k$. Since $T_k$ is a tree, there must be a path in $T_k$ that joins $v_k$ to $c'_k$. Now suppose $a'_k\neq b'_k\neq c'_k$. Since $T_k$ is a tree, there must be a path $P$ in $T_k$ from $a'_k$ to $b'_k$. Similarly, there must be a  path $P'$ that joins $c'_k$ to some vertex in $P$. Let $v_k$ be the vertex where these two paths meet. In each of these cases, we have three subpaths of $T_k$ that satisfy (2). Moreover, $R=\{v_1,\ldots,v_k\}$ also satisfies (1) since every vertex in $T_k$ is at a distance of at least $d$ from each member of $\{v_1,\ldots,v_{k-1}\}$. This contradicts the maximality of $k$ and proves the claim.
\end{subproof}

\begin{claim}
 \label{cla:ornamentationcircuits}
Every circuit of $\Or$ contains either the edge set of a cycle of $G$ or the edge set of a path in $G$ between two vertices in $R$.
\end{claim}

\begin{subproof}

Suppose for a contradiction that $C$ is a circuit of $\Or$ that contains neither the edge set of a cycle of $G$ nor the edge set of a path in $G$ joining vertices in $R$. Then $C\cap E(G)$ must consist of the edge sets of vertex-disjoint subtrees $S_1,S_2,\ldots, S_n$ of $G$ such that no $S_i$ contains more than one vertex in $R$. Thus, $C\subseteq (\cup_{i=1}^{c}F_i)\cup(\cup_{i=1}^{n}E(S_i))$. However, we will show by induction  on $|\cup_{i=1}^{n}E(S_i)|$, that $(\cup_{i=1}^{c}F_i)\cup(\cup_{i=1}^{n}E(S_i))$ is an independent set. Since no pair of gadgets is represented by submatrices whose sets of rows intersect, $\cup_{i=1}^{c}F_i$ is an independent set in $\Or$. Thus, the result holds when $|\cup_{i=1}^{n}E(S_i)|=0$. Now, consider $(\cup_{i=1}^{c}F_i)\cup(\cup_{i=1}^{n}E(S_i))$ where $|\cup_{i=1}^{n}E(S_i)|=k>0$ and the result holds for $|\cup_{i=1}^{n}E(S_i)|<k$. Delete a pendant edge $e$ in some $S_i$. By the induction hypothesis, $(\cup_{i=1}^{c}F_i)\cup(\cup_{i=1}^{n}E(S_i))-\{e\}$ is an independent set in $\Or$. Since $e$ is a pendant edge in some $S_i$, it must be a coloop in $(\Or)|((\cup_{i=1}^{c}F_i)\cup(\cup_{i=1}^{n}E(S_i)))$. Thus, $(\cup_{i=1}^{c}F_i)\cup(\cup_{i=1}^{n}E(S_i))$ is an independent set in $\Or$. By contradiction, this proves the claim.
\end{subproof}

Let $M$ be the dual matroid of $\Or$, and let $\lambda_M$ and $\lambda_G$ be the connectivity functions of $M$ and $M(G)$, respectively. Then, by duality, $\lambda_{\Or}=\lambda_M$.

\begin{claim}
\label{cla:connected}
 The matroid $\Or$ is cyclically $k$-connected.
\end{claim}

\begin{subproof}
   Suppose for a contradiction that $(X,Y)$ is a cyclic $k'$-separation of $\Or$, where $k'<k$. Let $A\cup B=E(G)$, with $A\subseteq X$ and $B\subseteq Y$. Since $M(G)$ has cyclic connectivity $k>k'$, it has no cyclic $k'$-separation. Therefore, one of $A$ or $B$, say $A$, has no cycles. However, since $(X,Y)$ is a cyclic $k'$-separation, $X$ and $Y$ each contain a circuit of $\Or$. Since $A$, and therefore $X$, contain no edge set of a cycle of $G$, we see from Claim \ref{cla:ornamentationcircuits} that $X$, and therefore $A$, contain the edge set of a path in $G$ joining vertices in $R$. By Claim \ref{cla:R}, this path has length at least $d$. This path must be contained in some component of $G[A]$ with edge set $A_1$. If a cubic graph either is disconnected or has a cut vertex, then both sides of the separation must contain cycles. Therefore, $G$ is a connected graph with no cut vertices. Let $B_1=E(G)-A_1$, and let $A_2=A-A_1$. Suppose $G[B_1]$ is not connected. Then, since $G[A_1]$ is a tree, there is a unique path in $G[A_1]$ from one component of $G[B_1]$ to another. This implies that $G$ has a cut vertex. Thus, we deduce that $G[B_1]$ is connected.

Let $r_G$ be the rank function of $M(G)$. Since $B_1$ is the disjoint union of $B$ and $A_2$, we have $r_G(B_1)\leq r_G(B)+r_G(A_2)$. Moreover, since $G[A_1]$ is a component of $G[A]$, we have $r_G(A_1)=r_G(A)-r_G(A_2)$. Therefore, $\lambda_G(A_1)\leq r_G(A)-r_G(A_2)+r_G(B)+r_G(A_2)-r_G(E(G))=\lambda_G(A)$. Let $W$ be the set of vertices of the vertex boundary between $A_1$ and $B_1$.  We have $\lambda_G(A_1)=r_G(A_1)+r_G(B_1)-r_G(E(G))=|V(G[A_1])|-1+|V(G[B_1])|-1-(|V(G)|-1)=|W|-1$. Thus, we have $|W|-1=\lambda_G(A_1)\leq\lambda_G(A)\leq\lambda_M(X)<k'$. Therefore, $|W|<k'+1$.

Note that, since $G$ is cubic and $G[A_1]$ contains no cycle, $G[A_1]$ is a cubic tree whose set of leaves is $W$. We now claim that no vertex of $G[A_1]$ is at a distance greater than $\lfloor\log_2(k')\rfloor +1$ from $W$. Suppose for a contradiction that $v$ is such a vertex. Therefore, there are $3(2^{\lfloor\log_2(k')\rfloor+1})>3(2^{\log_2(k')})=3h$ vertices at distance $\lfloor\log_2(k')\rfloor+2$ from $v$ in $G[A_1]$. This implies that $|W|>3k'$, contradicting the facts that $|W|<k'+1$ and $k'$ is a positive integer.

Therefore, each vertex of $G[A_1]$ is a distance of at most $\lfloor\log_2(k')\rfloor+1$ from $W$. Thus, since $G$ is cubic, an upper bound for $|A_1|$ is \[3(\sum_{i=0}^{\lfloor\log_2(k')\rfloor}2^i)=3(2^{\lfloor\log_2(k')\rfloor+1}-1)\leq3(2^{\lfloor\log_2(k)\rfloor+1}-1)<d,\] contradicting the fact that $G[A_1]$ must have a path of length at least $d$. Thus, $\Or$ has no cyclic $k'$-separation for $k'<k$ and is therefore cyclically $k$-connected.
\end{subproof}

For each $v_i\in R$, let $L_i$ consist of the three edges in $H$ that join $\phi(v_i)$ to the vertices in $\{a_i,b_i,c_i\}$. Let $D'$ consist of all the edges in $H$ incident with a vertex in $\phi(R)$ other than the edges in some $L_i$. 

\begin{claim}
\label{cla:N}
Consider $c$ copies of the submatrix $\begin{bmatrix}
 1&0&1\\
-1&1&0\\
0&-1&-1\\
\end{bmatrix}$ of the signed incidence matrix of $K_{m-c}$, where each of these submatrices has rows indexed by some $\{a_i,b_i,c_i\}$. Then $\Or$ has a minor $N$ that is the vector matroid of the matrix obtained from the signed incidence matrix of $K_{m-c}$ by altering each of these submatrices to become the following matrix, where the bottom row is a new row added to the original matrix. Here, $F_i=\{d_i,e_i,f_i,g_i\}$ is a gadget.
\[\begin{blockarray}{cccccccc}
&&&& d_i & e_i & f_i & g_i\\
\begin{block}{c[ccccccc]}
a_i&1&0&1&0&1&0&1\\
b_i&-1&1&0&0&1&1&0\\
c_i&0&-1&-1&0&0&1&1\\
&0&0&0&1&1&1&1\\
\end{block}
\end{blockarray}\] 
\end{claim}

\begin{subproof}
Recall that $C$ and $D$ are the sets of edges such that $G/C\backslash D=H\cong K_m$. Then $N=(\Or)/C\backslash(D\cup D')/\cup_{i=1}^{c} L_i$. Informally, $N$ is the result of ``gluing'' each $F_i$ onto the set $\{a_i,b_i,c_i\}$ of vertices in $K_{m-c}$.
\end{subproof}

Call the resulting matrix $J$ so that $N=M(J)$. Note that $J$ has $m$ rows and $r(N)=m-1$. Let $\Np$ be a rank-$(\leq 2t)$ perturbation of $N$.

\begin{claim}
\label{cla:Np2t}
 For some $s\leq 2t$, there are vectors $w''_1,\ldots, w''_s$ and a submatrix $J'$ of $J$ such that $\Np$ has a minor isomorphic to the vector matroid of \begin{center}
\begin{tabular}{|c|}
\hline
\phantom{XXXXXXXXX}$w''_1$\phantom{XXXXXXXXX}\\
\phantom{XXXXXXXXX}$\vdots$\phantom{XXXXXXXXX}\\
\phantom{XXXXXXXXX}$w''_s$\phantom{XXXXXXXXX}\\
\hline
\multirow{2}{*}{\phantom{XXXXXXXXX}$J'$\phantom{XXXXXXXXX}}\\
\\
\hline
\end{tabular}
\end{center} and such that $J'$ contains at least $20(3^{16t})$ copies of the submatrix \[\begin{blockarray}{ccccc}
& d_i & e_i & f_i & g_i\\
\begin{block}{c[cccc]}
a_i&0&1&0&1\\
b_i&0&1&1&0\\
c_i&0&0&1&1\\
&1&1&1&1\\
\end{block}
\end{blockarray}\] that represents a gadget.
\end{claim}

\begin{subproof}
By Remark \ref{contract}, $\Np$ is the result of contracting $C'$ from the vector matroid of the matrix below, where $\Delta$ is some arbitrary ternary matrix.
\begin{center}
\begin{tabular}{|c|c|}
\multicolumn{1}{c}{}&\multicolumn{1}{c}{$C'$}\\
\hline
\phantom{XXXXXXXXX}$w_1$\phantom{XXXXXXXXX}&\\
\phantom{XXXXXXXXX}$\vdots$\phantom{XXXXXXXXX}&$-I$\\
\phantom{XXXXXXXXX}$w_{2t}$\phantom{XXXXXXXXX}&\\
\hline
\multirow{2}{*}{\phantom{XXXXXXXXX}$J$\phantom{XXXXXXXXX}}&\multirow{4}{*}{$\Delta$}\\
&\\
\cline{1-1}
\multirow{2}{*}{\phantom{XXXXXXXXX}$0$\phantom{XXXXXXXXX}}&\\
&\\
\hline
\end{tabular}
\end{center} 
Let $V_{\Delta}$ be the set of row indices of a basis for the rowspace of $\Delta$. Delete from $N$ all elements represented by columns with nonzero entries in $V_{\Delta}$, along with all gadgets containing such an element. This is equivalent to deleting vertices from the complete graph $K_{m-c}$ that was used to construct $N$, as well as any gadgets glued onto these vertices. Thus, we are still left with a complete graph with gadgets glued onto it. Moreover, since $|V_{\Delta}|\leq 2t$, we have at least $c-2t\geq20(3^{16t})$ gadgets remaining. Since $V_{\Delta}$ is a basis for the rowspace of $\Delta$, we may perform row operations to obtain the following matrix, where $J'$ is a submatrix of $J$, where each $w'_i$ is a coordinate projection of $w_i$, and where $\Delta'=\Delta[V_{\Delta},C']$.

\begin{center}
\begin{tabular}{|c|c|}
\multicolumn{1}{c}{}&\multicolumn{1}{c}{$C'$}\\
\hline
\phantom{XXXXXXXXX}$w'_1$\phantom{XXXXXXXXX}&\\
\phantom{XXXXXXXXX}$\vdots$\phantom{XXXXXXXXX}&$-I$\\
\phantom{XXXXXXXXX}$w'_{2t}$\phantom{XXXXXXXXX}&\\
\hline
\multirow{2}{*}{\phantom{XXXXXXXXX}$J'$\phantom{XXXXXXXXX}}&\multirow{2}{*}{$0$}\\
&\\
\hline
\multirow{2}{*}{\phantom{XXXXXXXXX}$0$\phantom{XXXXXXXXX}}&\multirow{2}{*}{$\Delta'$}\\
&\\
\hline
\end{tabular}
\end{center}

For each element of $V_{\Delta}$, we contract one element of $C'$, pivoting on an entry in the row of $\Delta'$. We obtain the following matrix.
\begin{center}
\begin{tabular}{|c|c|}
\multicolumn{1}{c}{}&\multicolumn{1}{c}{$C''$}\\
\hline
\phantom{XXXXXXXXX}$w'_1$\phantom{XXXXXXXXX}&\\
\phantom{XXXXXXXXX}$\vdots$\phantom{XXXXXXXXX}&$Q$\\
\phantom{XXXXXXXXX}$w'_{2t}$\phantom{XXXXXXXXX}&\\
\hline
\multirow{2}{*}{\phantom{XXXXXXXXX}$J'$\phantom{XXXXXXXXX}}&\multirow{2}{*}{$0$}\\
&\\
\hline
\end{tabular}
\end{center}
By contracting $C''$, we obtain the desired matrix, with $s=2t-|C''|$.
\end{subproof}

Recall that $\{a_0,b_0,c_0\}$, $\{a_1,b_1,c_1\}$, $\ldots$, $\{a_c,b_c,c_c\}$ are the sets of vertices of $H\cong K_m$ onto which the gadgets are ``glued''. Figure \ref{fig:M(J'')}, \begin{figure}
\[\begin{tikzpicture}[x=.5cm, y=.5cm]
	\coordinate (c1j) at (1,9);
	\coordinate (b1j) at (2,10.5);
	\coordinate (d1j) at (3,7.8);
	\coordinate (c11) at (5,9);
	\coordinate (b11) at (6,10.5);
	\coordinate (d11) at (7,7.8);
	\coordinate (c0) at (8.5,9);
	\coordinate (b0) at (10,10.5);
	\coordinate (d0) at (11,7.8);
	\coordinate (c21) at (8.5,6);
	\coordinate (b21) at (10,7);
	\coordinate (d21) at (11,5);
	\coordinate (c2j) at (8.5,2);
	\coordinate (b2j) at (10,3);
	\coordinate (d2j) at (11,1);
	\coordinate (ci1) at (13,9);
	\coordinate (bi1) at (14,10.5);
	\coordinate (di1) at (15,7.8);
	\coordinate (cij) at (17,9);
	\coordinate (bij) at (18,10.5);
	\coordinate (dij) at (19,7.8);

\shade [left color=black, right color=white] (b1j) -- (c1j) -- (d1j) -- cycle;
\shade [left color=black, right color=white] (b11) -- (c11) -- (d11) -- cycle;
\shade [left color=black, right color=white] (bi1) -- (ci1) -- (di1) -- cycle;
\shade [left color=black, right color=white] (bij) -- (cij) -- (dij) -- cycle;
\shade [top color=black, bottom color=white] (b21) -- (c21) -- (d21) -- cycle;
\shade [top color=black, bottom color=white] (b2j) -- (c2j) -- (d2j) -- cycle;
	\vertex[fill] (c1j) at (1,9){};
	\vertex[fill] (b1j) at (2,10.5) [label=above:$F_{1,5(3^{8t})}$] {};
	\vertex[fill] (d1j) at (3,7.8){};
	\vertex[fill] (c11) at (5,9){};
	\vertex[fill] (b11) at (6,10.5) [label=above:$F_{1,1}$] {};
	\vertex[fill] (d11) at (7,7.8){};
	\vertex[fill] (c0) at (8.5,9) [label=above:$b_0$] {};
	\vertex[fill] (b0) at (10,10.5) [label=above:$a_0$] {};
	\vertex[fill] (d0) at (11,7.8) [label=above:$c_0$] {};
	\vertex[fill] (c21) at (8.5,6) [label=left:$F_{2,1}$] {};
	\vertex[fill] (b21) at (10,7){};
	\vertex[fill] (d21) at (11,5){};
	\vertex[fill] (c2j) at (8.5,2) [label=left:$F_{2,5(3^{8t})}$] {};
	\vertex[fill] (b2j) at (10,3){};
	\vertex[fill] (d2j) at (11,1){};
	\vertex[fill] (ci1) at (13,9){};
	\vertex[fill] (bi1) at (14,10.5) [label=above:$F_{4(3^{8t}),1}$] {};
	\vertex[fill] (di1) at (15,7.8){};
	\vertex[fill] (cij) at (17,9){};
	\vertex[fill] (bij) at (18,10.5) [label=above:$F_{4(3^{8t}),5(3^{8t})}$] {};
	\vertex[fill] (dij) at (19,7.8){};

	\node[align=left] at (12,5) {$\iddots$};
	\node[align=left] at (13,5.8) {$\iddots$};
	\node[align=left] at (14,6.6) {$\iddots$};
	\node[align=left] at (15,7.4) {$\iddots$};
	\node[align=left] at (4,8.5) {$\dots$};
	\node[align=left] at (4,9.5) {$\dots$};
	\node[align=left] at (16,8.5) {$\dots$};
	\node[align=left] at (16,9.5) {$\dots$};
	\node[align=left] at (9.3,4) {$\vdots$};
	\node[align=left] at (10.5,4) {$\vdots$};
	\path
		(b1j) edge (b11)
		(b11) edge (b0)
		(c1j) edge (c11)
		(c11) edge (c0)
		(d1j) edge (d11)
		(d11) edge (d0)
		(b2j) edge (b21)
		(b21) edge (b0)
		(c2j) edge (c21)
		(c21) edge (c0)
		(d2j) edge (d21)
		(d21) edge (d0)
		(bij) edge (bi1)
		(bi1) edge (b0)
		(cij) edge (ci1)
		(ci1) edge (c0)
		(dij) edge (di1)
		(di1) edge (d0)
	;
\end{tikzpicture}\]
\caption{A representation of the matroid $M(J'')$ in Claim \ref{cla:J''}}
  \label{fig:M(J'')}
\end{figure}
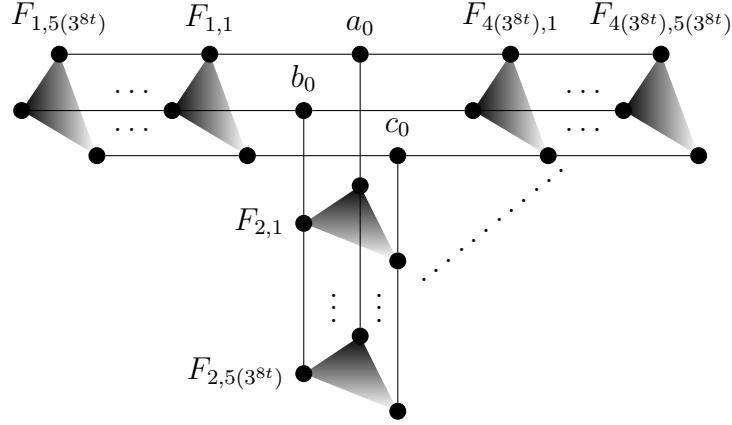 shows a representation of a restriction of $M(J')$ where each shaded triangle represents a gadget $F_i$ with vertices $a_i$, $b_i$, and $c_i$ positioned at the top, left, and bottom respectively. Consider the subtree of $K_{m-c}$ obtained by deleting from the matroid represented in Figure \ref{fig:M(J'')} all of the gadgets as well as all vertices $b_i$ and $c_i$. Call this tree $T_a$. Figure \ref{fig:T_a} 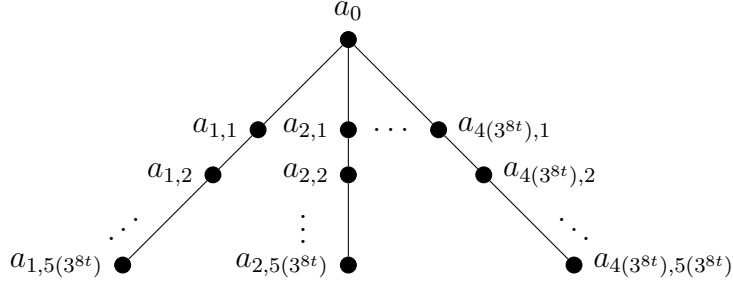
\begin{figure}
\[\begin{tikzpicture}[x=.6cm, y=.6cm]
	\vertex[fill] (a0) at (6,6) [label=above:$a_0$] {};
	\vertex[fill] (a11) at (4,4) [label=left:$a_{1,1}$] {};
	\vertex[fill] (a21) at (6,4) [label=left:$a_{2,1}$] {};
	\vertex[fill] (ai1) at (8,4) [label=right:$a_{4(3^{8t}),1}$] {};
	\vertex[fill] (a12) at (3,3) [label=left:$a_{1,2}$] {};
	\vertex[fill] (a22) at (6,3) [label=left:$a_{2,2}$] {};
	\vertex[fill] (ai2) at (9,3) [label=right:$a_{4(3^{8t}),2}$] {};
	\vertex[fill] (a1j) at (1,1) [label=left:$a_{1,5(3^{8t})}$] {};
	\vertex[fill] (a2j) at (6,1) [label=left:$a_{2,5(3^{8t})}$] {};
	\vertex[fill] (aij) at (11,1) [label=right:$a_{4(3^{8t}),5(3^{8t})}$] {};

	\node[align=left] at (1,2) {$\iddots$};
	\node[align=left] at (5,2) {$\vdots$};
	\node[align=left] at (7,4) {$\dots$};
	\node[align=left] at (11,2) {$\ddots$};
	\path
		(a0) edge (a11)
		(a0) edge (a21)
		(a0) edge (ai1)
		(a12) edge (a11)
		(a22) edge (a21)
		(ai2) edge (ai1)
		(a12) edge (a1j)
		(a22) edge (a2j)
		(ai2) edge (aij)
	;
\end{tikzpicture}\]
\caption{The tree $T_a$ with its vertex labels, used in Claim \ref{cla:J''}}
  \label{fig:T_a}
\end{figure} shows $T_a$; the trees $T_b$ and $T_c$ are defined similarly.

\begin{claim}
\label{cla:J''}
The matrix $J'$ has a submatrix $J''$, with the same number of rows as $J'$, such that $M(J'')$ is represented by Figure \ref{fig:M(J'')}.
\end{claim}

\begin{subproof}
Partition the set of gadgets into $4(3^{8t})$ subsets $\mathcal{F}_1,\ldots,\mathcal{F}_{4(3^{8t})}$, each of size at least $5(3^{8t})$. This is possible since $20(3^{16t})=(5)(3^{8t})(4)(3^{8t})$. Let $F_{i,j}=\{d_{i,j},e_{i,j},f_{i,j},g_{i,j}\}$ be the $j$-th gadget in $\mathcal{F}_i$, and let it be glued onto the vertices $\{a_{i,j},b_{i,j},c_{i,j}\}$.

Consider the submatrix $J''$ of $J'$ consisting of the columns indexed by the union of $E(T_a)$, $E(T_b)$, and $E(T_c)$ with the union $F$ of all of the gadgets. One can see that $M(J'')$ can be represented by Figure \ref{fig:M(J'')}.
\end{subproof}

\begin{claim}
\label{cla:N'}
There are ternary matrices $U$ and $J'''$ such that $M(J''')$ can be represented by Figure \ref{fig:M(J''')} 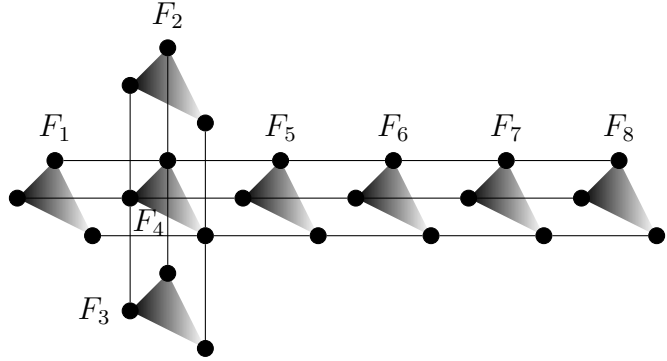
\begin{figure}
\[\begin{tikzpicture}[x=.5cm, y=.5cm]
	\coordinate (b1) at (2,6);
	\coordinate (c1) at (1,5);
	\coordinate (d1) at (3,4);
	\coordinate (b2) at (5,9);
	\coordinate (c2) at (4,8);
	\coordinate (d2) at (6,7);
	\coordinate (b3) at (5,3);
	\coordinate (c3) at (4,2);
	\coordinate (d3) at (6,1);
	\coordinate (b4) at (5,6);
	\coordinate (c4) at (4,5);
	\coordinate (d4) at (6,4);
	\coordinate (b5) at (8,6);
	\coordinate (c5) at (7,5);
	\coordinate (d5) at (9,4);
	\coordinate (b6) at (11,6);
	\coordinate (c6) at (10,5);
	\coordinate (d6) at (12,4);
	\coordinate (b7) at (14,6);
	\coordinate (c7) at (13,5);
	\coordinate (d7) at (15,4);
	\coordinate (b8) at (17,6);
	\coordinate (c8) at (16,5);
	\coordinate (d8) at (18,4);

	\shade [left color=black, right color=white] (b1) -- (c1) -- (d1) -- cycle;
	\shade [left color=black, right color=white] (b2) -- (c2) -- (d2) -- cycle;
	\shade [left color=black, right color=white] (b3) -- (c3) -- (d3) -- cycle;
	\shade [left color=black, right color=white] (b4) -- (c4) -- (d4) -- cycle;
	\shade [left color=black, right color=white] (b5) -- (c5) -- (d5) -- cycle;
	\shade [left color=black, right color=white] (b6) -- (c6) -- (d6) -- cycle;
	\shade [left color=black, right color=white] (b7) -- (c7) -- (d7) -- cycle;
	\shade [left color=black, right color=white] (b8) -- (c8) -- (d8) -- cycle;

	\vertex[fill] (b1) at (2,6)[label=above:$F_1$] {};
	\vertex[fill] (c1) at (1,5){};
	\vertex[fill] (d1) at (3,4){};
	\vertex[fill] (b2) at (5,9)[label=above:$F_2$] {};
	\vertex[fill] (c2) at (4,8){};
	\vertex[fill] (d2) at (6,7){};
	\vertex[fill] (b3) at (5,3){};
	\vertex[fill] (c3) at (4,2)[label=left:$F_3$] {};
	\vertex[fill] (d3) at (6,1){};
	\vertex[fill] (b4) at (5,6){};
	\vertex[fill] (c4) at (4,5){};
	\vertex[fill] (d4) at (6,4){};
	\vertex[fill] (b5) at (8,6)[label=above:$F_5$] {};
	\vertex[fill] (c5) at (7,5){};
	\vertex[fill] (d5) at (9,4){};
	\vertex[fill] (b6) at (11,6)[label=above:$F_6$] {};
	\vertex[fill] (c6) at (10,5){};
	\vertex[fill] (d6) at (12,4){};
	\vertex[fill] (b7) at (14,6)[label=above:$F_7$] {};
	\vertex[fill] (c7) at (13,5){};
	\vertex[fill] (d7) at (15,4){};
	\vertex[fill] (b8) at (17,6)[label=above:$F_8$] {};
	\vertex[fill] (c8) at (16,5){};
	\vertex[fill] (d8) at (18,4){};

	\node[align=left] at (4.5,4.35) {$F_4$};

	\path
		(b1) edge (b4)
		(b2) edge (b4)
		(b3) edge (b4)
		(b4) edge (b5)
		(b5) edge (b6)
		(b6) edge (b7)
		(b7) edge (b8)
		(c1) edge (c4)
		(c2) edge (c4)
		(c3) edge (c4)
		(c4) edge (c5)
		(c5) edge (c6)
		(c6) edge (c7)
		(c7) edge (c8)
		(d1) edge (d4)
		(d2) edge (d4)
		(d3) edge (d4)
		(d4) edge (d5)
		(d5) edge (d6)
		(d6) edge (d7)
		(d7) edge (d8)
;
\end{tikzpicture}\]
\caption{A representation of the matroid $M(J''')$ in Claim \ref{cla:N'}}
  \label{fig:M(J''')}
\end{figure} and such that $\Np$ has a minor $N'$ represented by the following matrix. \begin{center}
\begin{tabular}{|c|c|c|c|c|}
\multicolumn{1}{c}{}&\multicolumn{1}{c}{$F_1$}&\multicolumn{1}{c}{$F_2$}&\multicolumn{1}{c}{$\cdots$}&\multicolumn{1}{c}{$F_8$}\\
\hline
$0$&$U$&$U$&$\cdots$&$U$\\
\hline
\multicolumn{5}{|c|}{$J'''$}\\
\hline
\end{tabular}
\end{center}
\end{claim}

\begin{subproof}
Let $W=\begin{bmatrix}
\phantom{XX}w''_1\phantom{XX}\\
\vdots\\
w''_s\\
\end{bmatrix}$. Since $E(T_a)\cup E(T_b)\cup E(T_c)$ is an independent set in $M(J'')$, we may perform row operations so that the portion of $W$ with columns indexed by $E(T_a)\cup E(T_b)\cup E(T_c)$ becomes the zero matrix. Thus, we have the following matrix.
\begin{center}
\begin{tabular}{|ccc|c|}
\multicolumn{1}{c}{$E(T_a)$}&\multicolumn{1}{c}{$E(T_b)$}&\multicolumn{1}{c}{$E(T_c)$}&\multicolumn{1}{c}{$F$}\\
\hline
\multicolumn{3}{|c|}{$0$}&\phantom{XXXXXXX}$W'$\phantom{XXXXXXX}\\
\hline
\multicolumn{4}{|c|}{$J''$}\\
\hline
\end{tabular}
\end{center}

The portion of $W'$ whose columns are indexed by the elements of a gadget is an $s\times4$ ternary matrix; therefore, there are $3^{4s}\leq3^{8t}$ possible such matrices. Since each $\mathcal{F}_i$ contains at least $5(3^{8t})$ gadgets, the pigeonhole principle implies that each $\mathcal{F}_i$ contains five gadgets whose corresponding submatrices of $W'$ are equal. Again by the pigeonhole principle, since there are $4(3^{8t})$ sets $\mathcal{F}_i$, there is a set of four $\mathcal{F}_i$ such that each contains at least five gadgets such that all 20 of the gadgets correspond to equal submatrices of $W'$.

Delete all of the other $c-2t-20$ gadgets. All of the remaining gadgets come from four $\mathcal{F}_i$ which we can relabel as $\mathcal{F}_1$, $\mathcal{F}_2$, $\mathcal{F}_3$, and $\mathcal{F}_4$. In addition, delete all but one gadget from each of $\mathcal{F}_2$, $\mathcal{F}_3$, and $\mathcal{F}_4$, cosimplify the resulting matroid, and contract the remaining edges incident with either $a_0$, $b_0$, or $c_0$. In the resulting matroid, there are eight remaining gadgets which we relabel as $F_1,F_2,\ldots, F_8$. The elements of each $F_i$ we relabel as $\{d_i,e_i,f_i,g_i\}$, and we relabel the vertices onto which $F_i$ is glued as $a_i$, $b_i$, and $c_i$. This matroid is the desired matroid $N'$. In Figure \ref{fig:M(J''')}, again each shaded triangle represents a gadget $F_i$ with vertices $a_i$, $b_i$, and $c_i$ positioned at the top, left, and bottom respectively.
\end{subproof}

\begin{claim}
\label{cla:nopert}
Regardless of $U$, the matroid $N'$ is not a frame matroid.
\end{claim}

\begin{subproof}
Let $P$ consist of all edges joining a gadget $F_i$ to a gadget $F_{i+1}$, for $4\leq i\leq 7$,  and for $1\leq i\leq 3$, let $\alpha_i$, $\beta_i$, and $\gamma_i$ be the edges that join $a_i$, $b_i$, and $c_i$ to $a_4$, $b_4$, and $c_4$, respectively. Now let $N''$ be the simplification of $N'/P/\{\alpha_1,\beta_2,\gamma_3\}/(\cup_{i=1}^{i=3}\{e_i,f_i,g_i\})/\{d_5,e_6,f_7,g_8\}$. In the case where $U$ is the zero matrix, $N''$ is the generalized parallel connection of $M(K_5)$ with the ternary Dowling geometry of rank $3$; that is, $N''$ is the vector matroid of the following matrix, where the last six columns come from the gadgets $F_5,\ldots,F_8$.
\[\begin{blockarray}{ccccccccccccc}
d_1 & d_2 & d_3 & d_4 & e_4 & f_4 & g_4&&&&&&\\
\begin{block}{[ccccccccccccc]}
1&0&0&0&1&0&1&1&1&0&0&1&1\\
0&1&0&0&1&1&0&-1&1&1&1&0&0\\
0&0&1&0&0&1&1&0&0&-1&1&-1&1\\
0&0&0&1&1&1&1&0&0&0&0&0&0\\
\end{block}
\end{blockarray}\]

In the Appendix, we show how the mathematics software system SageMath was used to show that, regardless of $U$, the matroid $N''$, and therefore $N'$, are not signed-graphic matroids. The computations were carried out in Version 8.0 of SageMath \cite{sage}, in particular making use of the \emph{matroids} component \cite{sage-matroid}. We used the CoCalc (formerly SageMathCloud) online interface.

Indeed, since $U$ has four columns, its rank is at most 4. Therefore, we may assume that $U$ has at most four rows. There are 16 possible bases for $M(U)$ -- one of size 0, four of size 1, six of size 2, four of size 3, and one of size 4. For each of these bases, we checked all possible matrices $U$ where the basis indexed an identity matrix, unless the resulting matroid $M(U)$ contained a basis that was already checked. In each case, $N''$ was found not to be signed-graphic. A ternary matroid is a frame matroid if and only if it is a signed-graphic matroid. Therefore, $N'$ is not a frame matroid.
\end{subproof}

Recall that $M^*=\Or$.

\begin{claim}
\label{cla:Mnotdualpert}
The matroid $M$ is not a rank-$(\leq t)$ perturbation of the dual of a frame matroid.
\end{claim}

\begin{subproof}
Suppose otherwise. Then by Lemma \ref{dualperturbations}, $\Or$ is a rank-$(\leq 2t)$ perturbation of a frame matroid. The class of matroids that are rank-$(\leq 2t)$ perturbations of a frame matroid is minor-closed. Therefore, by Claim \ref{cla:N}, $N$ is a rank-$(\leq 2t)$ perturbation of a frame matroid. However, by Claims \ref{cla:N'} and \ref{cla:nopert}, this is impossible.
\end{subproof}

\begin{claim}
\label{cla:Mnotpert}
The matroid $M$ is not a rank-$(\leq t)$ perturbation of a frame matroid.
\end{claim}

\begin{subproof}
Suppose otherwise. Then by Lemma \ref{dualperturbations}, $\Or$ is a rank-$(\leq 2t)$ perturbation of the dual of a frame matroid. Recall that, by Theorem \ref{thomassen}, $G$ contains a minor isomorphic to $K_m$. Therefore, by Lemma \ref{lem:dyadic}, $\Or$ has a minor isomorphic to $M(K_m)$. But, since $m\geq g(3,2t)$, Lemma \ref{graphicvsframedual} implies that $M(K_m)$, and therefore $\Or$, are not rank-$(\leq 2t)$ perturbations of the dual of a frame matroid.
\end{subproof}

By Lemma \ref{lem:dyadic}, $\Or$, is dyadic. Since the class of dyadic matroids is closed under duality, $M$ is dyadic also. By Claim \ref{cla:connected} and duality, $M$ is vertically $k$-connected. Claims \ref{cla:Mnotdualpert} and \ref{cla:Mnotpert} show that $M$ is not a rank-$(\leq t)$ perturbation of either a frame matroid or the dual of a frame matroid. This completes the proof of the theorem.
\end{proof}

\begin{corollary}
 \label{cor:counterexample}
The family of matroids given in Theorem \ref{thm:construction} is a counterexample to Conjecture \ref{con:perturb}.
\end{corollary}

\begin{proof}
 By Theorem \ref{thm:construction}, for every $k,t\in\mathbb{Z}_+$, there exists a vertically $k$-connected dyadic matroid that is not a rank-$(\leq t)$ perturbation of either a frame matroid or the dual of a frame matroid. Thus, neither (i) nor (ii) of Conjecture \ref{con:perturb} is satisfied. Moreover, since the matroids given by Theorem \ref{thm:construction} are dyadic, they are representable over $\mathrm{GF}(3)$ which has no proper subfield. Therefore, if $\mathbb{F}=\mathrm{GF}(3)$ (or any prime field of odd order, for that matter), then the matroids given by Theorem \ref{thm:construction} do not satisfy (iii) of Conjecture \ref{con:perturb} either.
\end{proof}

\begin{remark}
Our construction relies heavily on a non-standard frame matroid representation of $M(K_4)$, and involves a notion of 4-sums. Each gadget is $4$-separating in our construction. The following result by Zaslavsky \cite{z90} shows that 5-sums and higher cannot be encountered in an analogous way.
\end{remark}

\begin{theorem}[{\cite[Proposition 5A]{z90}}]\label{thm:zaslavsky}
    Let $\Omega$ be a biased graph such that the frame matroid of $\Omega$ is isomorphic to $M(K_m)$ for $m \geq 5$. Then $\Omega$ is isomorphic to either $(K_m, \emptyset)$ or $\Phi_{m-1}'$, where the latter is the biased graph obtained by adding an edge $e$ in parallel with an edge of $K_m$, taking the unbalanced cycles to be the collection of cycles through $e$, and contracting $e$ in the resulting biased graph.
\end{theorem}

This makes us cautiously optimistic that our construction cannot be generalized to have ``gadgets'' with arbitrary connectivity.

We believe that the subfield case, as stated by Geelen, Gerards, and Whittle \cite{ggw15}, does not need to be modified.

\section{Discussion}
\label{sec:discussion}

\subsection{Frame templates}
\label{sub:templates}
In Subsection \ref{sub:updated_conjectures}, we will offer some updated conjectures to replace Conjecture \ref{con:perturb}. Some of the hypotheses will be stated in terms of frame templates. Geelen, Gerards, and Whittle introduced the notions of \emph{subfield templates} and \emph{frame templates} in ~\cite{ggw15}. A template is a concise description of certain perturbations of represented matroids. We will recall several definitions concerning frame templates which essentially can be found in ~\cite{ggw15} as well as ~\cite{gvz17} and ~\cite{nw17}.

Let $A$ be a matrix over a field $\mathbb{F}$. Then $A$ is a \textit{frame matrix} if each column of $A$ has at most two nonzero entries. Let $\mathbb{F}^{\times}$ denote the multiplicative group of $\mathbb{F}$, and let $\Gamma$ be a subgroup of $\mathbb{F}^{\times}$. A $\Gamma$-frame matrix is a frame matrix $A$ such that:
\begin{itemize}
 \item Each column of $A$ with a nonzero entry contains a 1.
 \item If a column of $A$ has a second nonzero entry, then that entry is $-\gamma$ for some $\gamma\in\Gamma$.
\end{itemize}

A \textit{frame template} over $\mathbb{F}$ is a tuple $\Phi=(\Gamma,C,X,Y_0,Y_1,A_1,\Delta,\Lambda)$ such that the following hold\footnote{The authors of ~\cite{ggw15} divided our set $X$ into two separate sets which they called $X$ and $D$. Their set $X$ can be absorbed into $Y_0$, therefore we omit it.}:
\begin{itemize}
 \item [(i)] $\Gamma$ is a subgroup of $\mathbb{F}^{\times}$.
 \item [(ii)] $C$, $X$, $Y_0$ and $Y_1$ are disjoint finite sets.
 \item [(iii)] $A_1\in \mathbb{F}^{X\times (C\cup Y_0\cup Y_1)}$.
 \item [(iv)] $\Lambda$ is a subgroup of the additive group of $\mathbb{F}^X$ and is closed under scaling by elements of $\Gamma$.
 \item [(v)] $\Delta$ is a subgroup of the additive group of $\mathbb{F}^{C\cup Y_0 \cup Y_1}$ and is closed under scaling by elements of $\Gamma$.
\end{itemize}

Let $\Phi=(\Gamma,C,X,Y_0,Y_1,A_1,\Delta,\Lambda)$ be a frame template. Let $B$ and $E$ be finite sets, and let $A'\in\mathbb{F}^{B\times E}$. We say that $A'$ \textit{respects} $\Phi$ if the following hold:
\begin{itemize}
 \item [(i)] $X\subseteq B$ and $C, Y_0, Y_1\subseteq E$.
 \item [(ii)] $A'[X, C\cup Y_0\cup Y_1]=A_1$.
 \item [(iii)] There exists a set $Z\subseteq E-(C\cup Y_0\cup Y_1)$ such that $A'[X,Z]=0$, each column of $A'[B-X,Z]$ is a unit vector, and $A'[B-X, E-(C\cup Y_0\cup Y_1\cup Z)]$ is a $\Gamma$-frame matrix.
 \item [(iv)] Each column of $A'[X,E-(C\cup Y_0\cup Y_1\cup Z)]$ is contained in $\Lambda$.
 \item [(v)] Each row of $A'[B-X, C\cup Y_0\cup Y_1]$ is contained in $\Delta$.
\end{itemize}

The structure of $A'$ is shown below.

\begin{center}
\begin{tabular}{ r|c|c|ccc| }
\multicolumn{2}{c}{}&\multicolumn{1}{c}{$Z$}&\multicolumn{1}{c}{$Y_0$}&\multicolumn{1}{c}{$Y_1$}&\multicolumn{1}{c}{$C$}\\
\cline{2-6}
$X$&columns from $\Lambda$&$0$&&$A_1$&\\
\cline{2-6}
&$\Gamma$-frame matrix&unit columns&\multicolumn{3}{c|}{rows from  $\Delta$}\\
\cline{2-6}
\end{tabular}
\end{center}

Now, suppose that $A'$ respects $\Phi$ and that $A\in \mathbb{F}^{B\times E}$ satisfies the following conditions:
\begin{itemize}
\item [(i)] $A[B,E-Z]=A'[B,E-Z]$.
\item [(ii)] For each $i\in Z$ there exists $j\in Y_1$ such that the $i$-th column of $A$ is the sum of the $i$-th and the $j$-th columns of $A'$.
\end{itemize}
We say that such a matrix $A$ \textit{conforms} to $\Phi$.

Let $M$ be an $\mathbb{F}$-represented matroid. We say that $M$ \textit{conforms} to $\Phi$ if there is a matrix $A$ conforming to $\Phi$ such that $M$ is isomorphic to $M(A)/C\backslash Y_1$. We denote by $\mathcal{M}(\Phi)$ the set of $\mathbb{F}$-represented matroids that conform to $\Phi$.

\subsection{Updated conjectures}
\label{sub:updated_conjectures}
We offer several updated conjectures. We will state them as hypotheses for easy reference in other papers. We will give both a ``perturbation version'' and a ``template version'' of each hypothesis. The template version of Conjecture \ref{con:perturb} follows as Conjecture \ref{con:template}. Since a template gives a description of certain types of perturbations, Conjecture \ref{con:template} is false, as well as Conjecture \ref{con:perturb}. If $M$ is a represented matroid, we denote by $\widetilde M$ the matroid (in the usual sense) that arises from $M$. For a field $\mathbb{F}$ of characteristic $p\neq0$, we denote the prime subfield of $\mathbb{F}$ by $\bFp$.
\begin{conjecture}[{\cite[{Theorem 4.2}]{ggw15}}]
 \label{con:template}
Let $\mathbb F$ be a finite field, let $m$ be a positive integer, and let 
$\mathcal M$ be a minor-closed class of $\mathbb F$-represented matroids.
Then there exist $k\in\mathbb{Z}_+$ and
frame templates $\Phi_1,\ldots,\Phi_s,\Psi_1,\ldots,\Psi_t$ such that
\begin{itemize}
\item
$\mathcal{M}$ contains each of the classes
$\mathcal{M}(\Phi_1),\ldots,\mathcal{M}(\Phi_s)$,
\item
$\mathcal{M}$ contains the duals of the represented matroids in each of the classes
$\mathcal{M}(\Psi_1),\ldots,\mathcal{M}(\Psi_t)$, and
\item
if $M$ is a simple vertically $k$-connected member of $\mathcal M$
and $\widetilde M$ has no $PG(m-1,\bFp)$-minor, 
then either
$M$ is a member of at least one of the classes
$\mathcal{M}(\Phi_1),\ldots,\mathcal{M}(\Phi_s)$, or
$M^*$ is a member of at least one of the classes
$\mathcal{M}(\Psi_1),\ldots,\mathcal{M}(\Psi_t)$.
\end{itemize}
\end{conjecture}
Since every represented matroid conforming to a particular template is a bounded-rank perturbation of a represented frame matroid, our construction disproves Conjecture \ref{con:template} also.

The next hypothesis replaces the requirement of vertical connectivity with the stronger requirement of Tutte connectivity.

\begin{hypothesis}\label{hyp:connected}
    Let $\mathbb{F}$ be a finite field, and let $\mathcal{M}$ be a proper minor-closed class of $\mathbb{F}$-represented matroids. There exist constants $k,t\in\mathbb{Z}_+$ such that each $k$-connected member of $\mathcal{M}$ is a rank-$(\leq t)$ perturbation of an $\mathbb{F}$-represented matroid $N$, such that either
  \begin{enumerate}
      \item $N$ is a represented frame matroid,
      \item $N^*$ is a represented frame matroid, or
      \item $N$ is confined to a proper subfield of $\mathbb{F}$.
  \end{enumerate}
\end{hypothesis}

The template version of Hypothesis \ref{hyp:connected} follows.
\begin{hypothesis}
\label{hyp:connected-template}
Let $\mathbb F$ be a finite field, let $m$ be a positive integer, and let 
$\mathcal M$ be a minor-closed class of $\mathbb F$-represented matroids.
Then there exist $k\in\mathbb{Z}_+$ and
frame templates $\Phi_1,\ldots,\Phi_s,\Psi_1,\ldots,\Psi_t$ such that
\begin{enumerate}
\item
$\mathcal{M}$ contains each of the classes
$\mathcal{M}(\Phi_1),\ldots,\mathcal{M}(\Phi_s)$,
\item
$\mathcal{M}$ contains the duals of the represented matroids in each of the classes
$\mathcal{M}(\Psi_1),\ldots,\mathcal{M}(\Psi_t)$, and
\item
if $M$ is a simple $k$-connected member of $\mathcal M$ with at least $2k$ elements
and $\widetilde M$ has no $PG(m-1,\bFp)$-minor, 
then either
$M$ is a member of at least one of the classes
$\mathcal{M}(\Phi_1),\ldots,\mathcal{M}(\Phi_s)$, or
$M^*$ is a member of at least one of the classes
$\mathcal{M}(\Psi_1),\ldots,\mathcal{M}(\Psi_t)$.
\end{enumerate}
\end{hypothesis}

Taken together, the next two hypotheses revise Conjecture \ref{con:perturb} by pairing the condition of vertical connectivity with its natural match of having a large clique minor and by pairing the dual condition of cyclic connectivity with the property of having a large coclique minor.

\begin{hypothesis}
\label{hyp:vertandclique}
   Let $\mathbb{F}$ be a finite field, and let $\mathcal{M}$ be a proper minor-closed class of $\mathbb{F}$-represented matroids. There exist constants $k,t,n\in\mathbb{Z}_+$ such that each vertically $k$-connected member of $\mathcal{M}$ containing a minor isomorphic to $M(K_n)$ is a rank-$(\leq t)$ perturbation of an $\mathbb{F}$-represented matroid $N$, such that either
  \begin{enumerate}
      \item $N$ is a represented frame matroid, or
      \item $N$ is confined to a proper subfield of $\mathbb{F}$.
  \end{enumerate}
\end{hypothesis}

\begin{hypothesis}
\label{hyp:cycandcoclique}
   Let $\mathbb{F}$ be a finite field, and let $\mathcal{M}$ be a proper minor-closed class of $\mathbb{F}$-represented matroids. There exist constants $k,t,n\in\mathbb{Z}_+$ such that each cyclically $k$-connected member of $\mathcal{M}$ containing a minor isomorphic to $M^*(K_n)$ is a rank-$(\leq t)$ perturbation of an $\mathbb{F}$-represented matroid $N$, such that either
  \begin{enumerate}
      \item $N^*$ is a represented frame matroid, or
      \item $N$ is confined to a proper subfield of $\mathbb{F}$.
  \end{enumerate}
\end{hypothesis}

Now we give the template version of Hypotheses \ref{hyp:vertandclique} and \ref{hyp:cycandcoclique}.

\begin{hypothesis}
 \label{hyp:cliquetemplate}
Let $\mathbb F$ be a finite field, let $m$ be a positive integer, and let $\mathcal M$ be a minor-closed class of $\mathbb F$-represented matroids. Then there exist $k,n\in\mathbb{Z}_+$ and frame templates $\Phi_1,\ldots,\Phi_s,\Psi_1,\ldots,\Psi_t$ such that
\begin{enumerate}
\item $\mathcal{M}$ contains each of the classes $\mathcal{M}(\Phi_1),\ldots,\mathcal{M}(\Phi_s)$,
\item $\mathcal{M}$ contains the duals of the represented matroids in each of the classes $\mathcal{M}(\Psi_1),\ldots,\mathcal{M}(\Psi_t)$,
\item if $M$ is a simple vertically $k$-connected member of $\mathcal M$ with an $M(K_n)$-minor but no $PG(m-1,\bFp)$-minor, then $M$ is a member of at least one of the classes $\mathcal{M}(\Phi_1),\ldots,\mathcal{M}(\Phi_s)$, and
\item if $M$ is a cosimple cyclically $k$-connected member of $\mathcal M$ with an $M^*(K_n)$-minor but no $PG(m-1,\bFp)$-minor, then $M^*$ is a member of at least one of the classes $\mathcal{M}(\Psi_1),\ldots,\mathcal{M}(\Psi_t)$.
\end{enumerate}
\end{hypothesis}

\subsection{Consequences}
\label{sub:consequences}
We conclude Section \ref{sec:discussion} by detailing the consequences of our results on previously published papers that were written based on \cite{ggw15}. It should be noted that our results in ~\cite{gvz17} remain correct insofar as they regard templates themselves, as opposed to the applications of Conjecture \ref{con:template}. This means that all of the proofs in \cite[Section 3]{gvz17} remain valid. In particular, the determination of the minimal nontrivial binary frame templates \cite[Theorem 3.19]{gvz17} is accurate. (However, \cite[Definition 3.20]{gvz17} is no longer useful to prove the results in \cite[Section 4]{gvz17}.)

A significant portion of our work in \cite{gvz17}, our forthcoming work, and the work of Nelson and Walsh \cite{nw17} involves the use of the structure theory of Geelen, Gerards, and Whittle to obtain results about the extremal functions (also called growth rate functions) of classes of represented matroids. We will prove that these results are true subject to Hypothesis \ref{hyp:cliquetemplate}. The \emph{extremal function} for a minor-closed class $\mathcal{M}$, denoted by $h_{\mathcal{M}}(r)$, is the function whose value at an integer $r\geq0$ is given by the maximum number of elements in a simple represented matroid in $\mathcal{M}$ of rank at most $r$. We will make use of several results in the literature. The first of these is the Growth Rate Theorem of Geelen, Kung, and Whittle \cite[Theorem 1.1]{gkw09}.

\begin{theorem}[Growth Rate Theorem]
\label{thm:growthrate}
If $\mathcal{M}$ is a nonempty minor-closed class of matroids, then there exists $c\in\mathbb{R}$ such that either:
\begin{itemize}
 \item[(1)] $h_{\mathcal{M}}(r)\leq cr$ for all $r$,
\item[(2)] $\binom{r+1}{2}\leq h_{\mathcal{M}}(r)\leq cr^2$ for all $r$ and $\mathcal{M}$ contains all graphic matroids,
\item[(3)] there is a prime-power $q$ such that $\frac{q^r-1}{q-1}\leq h_{\mathcal{M}}(r)\leq cq^r$ for all $r$ and $\mathcal{M}$ contains all $\mathrm{GF}(q)$-representable matroids, or
\item[(4)] $\mathcal{M}$ contains all simple rank-2 matroids.
\end{itemize}
\end{theorem}

If (2) of the previous theorem holds for $\mathcal{M}$, then $\mathcal{M}$ is \emph{quadratically dense}. If $M$ is a simple rank-$r$ matroid in $\mathcal{M}$ such that $\varepsilon(M)=h_{\mathcal{M}}(r)$, then we call $M$ an \emph{extremal matroid} of $\mathcal{M}$.

The proof of the Growth Rate Theorem was based on work in \cite{gw03} and \cite{gk09}. Specifically, \cite{gw03} contains the following result.

\begin{theorem}[{\cite[Theorem 1.1]{gw03}}]
 \label{thm:excludegraph}
For any finite field $\mathbb{F}$ and any graph $G$, there exists an integer $c$ such that, if $M$ is an $\mathbb{F}$-represented matroid with no $M(G)$-minor, then $\varepsilon(M)\leq cr(M)$.
\end{theorem}

Geelen and Nelson proved the next result. (In fact, their result is a bit more detailed, but the following result follows from theirs.)
\begin{theorem}[{\cite[Theorem 6.1]{gn15}}]
\label{thm:geelen-nelson}
Let $\mathcal{M}$ be a quadratically dense minor-closed class of matroids and let $p(x)$ be a real quadratic polynomial with positive leading coefficient. If $h_{\mathcal{M}}(n)>p(n)$ for infinitely many $n\in\mathbb{Z}^+$, then for all integers $r,s\geq1$ there exists a vertically $s$-connected matroid $M\in\mathcal{M}$ satisfying $\varepsilon(M)>p(r(M))$ and $r(M)\geq r$.
\end{theorem}

The next result is from Nelson and Walsh \cite{nw17}. It is accurate since it was proved independently of \cite{ggw15}.

\begin{lemma}[{\cite[Lemma 2.2]{nw17}}]
\label{lem:vertextremal}
Let $\mathbb{F}$ be a finite field, let $f(x)$ be a real quadratic polynomial with positive leading coefficient, and let $k\in\mathbb{N}_0$. If $\mathcal{M}$ is a restriction-closed class of $\mathbb{F}$-represented matroids and if, for all sufficiently large $n$, the extremal function of $\mathcal{M}$ at $n$ is given by $f(n)$, then for all sufficiently large $r$, every rank-$r$ matroid $M\in\mathcal{M}$ with $\varepsilon(M)=f(r)$ is vertically $k$-connected.
\end{lemma}

Nelson and Walsh ~\cite{nw17} define a pair of templates $\Phi,\Phi'$ to be \emph{equivalent} if $\mathcal{M}(\Phi)=\mathcal{M}(\Phi')$. This is the definition we will use in the remainder of this paper, although it should be noted that we defined a different notion of equivalence in ~\cite{gvz17}. A pair of templates that are equivalent in the sense of ~\cite{nw17} are also equivalent in the sense of ~\cite{gvz17}, but the converse is not true. Moreover, Nelson and Walsh gave Definition \ref{def:Y-reduced} and proved Lemma \ref{lem:all-Y-reduced} below. Lemma \ref{lem:all-Y-reduced} is a result about templates in their own right, rather than applications of Conjecture \ref{con:template}, and is true for that reason.

\begin{definition}
 \label{def:Y-reduced}
A frame template $\Phi=(\Gamma,C,X,Y_0,Y_1,A_1,\Delta,\Lambda)$ over $\mathbb{F}$ is \emph{$Y$-reduced} if $\Delta[C]=\Gamma(\mathbb{F}_p^C)$ and $\Delta[Y_0\cup Y_1]=\{0\}$, and there is a partition $(X_0,X_1)$ of $X$ for which $\mathbb{F}_p^{X_0}\subseteq\Lambda[X_0]$ and $\Lambda[X_1]=\{0\}$. We will call the partition $X=X_0\cup X_1$ the \emph{reduction partition} of $\Phi$.
\end{definition}

\begin{lemma}[{\cite[Lemma 5.5]{nw17}}]
 \label{lem:all-Y-reduced}
Every frame template is equivalent to a $Y$-reduced frame template.
\end{lemma}

The next lemma is an easy observation.

\begin{lemma}
 \label{lem:no_copies}
Every frame template is equivalent to a $Y$-reduced frame template such that no column of $A_1[X,Y_1]$ is contained in $\Lambda$.
\end{lemma}

\begin{proof}
By Lemma \ref{lem:all-Y-reduced}, every template is equivalent to some $Y$-reduced template $\Phi=(\Gamma,C,X,Y_0,Y_1,A_1,\Delta,\Lambda)$. Note that the role of $Y_1$ in matroids conforming to $\Phi$ is to construct the set $Z$. Every element of $Z$ indexes a column constructed by placing a column of $A_1[X,Y_1]$ on top of an identity column. If such a column is made from a column of $A_1[X,Y_1]$ that is a copy of an element of $\Lambda$, then the column can also be obtained in $E-(Z\cup C\cup Y_0\cup Y_1)$ by choosing an identity column for the portion of the column coming from the $\Gamma$-frame matrix. Thus, that element of $Y_1$ is unnecessary, and a template equivalent to $\Phi$ can be obtained from $\Phi$ by removing that element of $Y_1$.
\end{proof}

We will now show that, for all sufficiently large ranks, the extremal function for the set of matroids conforming to a frame template is given by a quadratic polynomial. We will call a largest simple matroid of a given rank that conforms to a template an \emph{extremal matroid} of the template.

\begin{lemma}
 \label{lem:polynomial}
Let $\Phi=(\Gamma,C,X,Y_0,Y_1,A_1,\Delta,\Lambda)$ be a $Y$-reduced frame template, with reduction partition $X=X_0\cup X_1$, such that no column of $A_1[X,Y_1]$ is contained in $\Lambda$. Let $|\widehat{Y_0}|$ denote the number of columns of $A_1[X,Y_0]$ that are not contained in $\Lambda$. Let $|\widehat{\Lambda}|$ denote the maximum number of nonzero elements of $\Lambda$ that pairwise are not scalar multiples of each other. And let $t$ denote the difference between $|X_1|$ and the rank of the matrix $A_1[X_1,C\cup Y_0\cup Y_1]$. If $r\geq 2|C|+|X|-t+2$, then the size of a rank-$r$ extremal matroid of $\Phi$ is $ar^2+br+c$, where
 \[a=\frac{1}{2}|\Gamma||\Lambda|,\]
\[b=\frac{1}{2}|\Gamma||\Lambda|(2|C|+2t-2|X|-1)+|\Lambda|+|Y_1|,\] and 
\[c=\frac{1}{2}(|C|+t-|X|)[|\Gamma||\Lambda|(|C|+t-|X|-1)+2|\Lambda|+2|Y_1|]+|\widehat{\Lambda}|+|\widehat{Y_0}|.\]
\end{lemma}

\begin{proof}
An extremal matroid $M$ of  $\Phi$ is obtained by contracting $C$ and deleting $Y_1$ from the vector matroid of some matrix $A$ that conforms to $\Phi$. Let $r_C=r_{M(A)}(C)$. Then $r(M(A)\backslash Y_1)=r+r_C$, and the number of rows of $A$ is $r+r_C+t$. We wish to calculate the largest possible size of a simple matroid of the form $M(A)\backslash Y_1$, where $A$ conforms to $\Phi$ and where $r(M(A)\backslash Y_1/C)=r$. Since $A$ has $r+r_C+t$ rows, the number of rows of the $\Gamma$-frame submatrix of $A$ is $r+r_C+t-|X|$, which we abbreviate as $n$. Let $n\geq1$. Thus, $A$ has at least $|X|+1$ rows and rank at least $|X|-t+1$, and $r\geq |X|-r_C-t+1$.

In $E-(Z\cup C\cup Y_0\cup Y_1)$, there are $|\Gamma||\Lambda|\binom{n}{2}$ distinct possible columns where the $\Gamma$-frame matrix has two nonzero entries per column. There are $|\Lambda|n$ distinct possible columns where the $\Gamma$-frame matrix has one nonzero entry per column. And there are $|\widehat{\Lambda}|$ distinct possible nonzero columns where the $\Gamma$-frame matrix is a zero column because including all of the elements of $\Lambda$ would result in a matroid that is not simple.

The size of $Z$ is at most $|Y_1|n$ since there are that many possible distinct possible columns.

The entire sets $C$ and $Y_0$ are always contained in $M(A)$, but if any columns of $A_1[X,Y_0]$ are contained in $\Lambda$, then the corresponding element of $E-(Z\cup C\cup Y_0\cup Y_1)$ must be deleted in order for the matroid to be simple. Therefore, adding together the elements of $E-(Z\cup C\cup Y_0\cup Y_1)$, the elements of $Z$, and the elements of $C\cup Y_0$, we see that \[\varepsilon(M(A)\backslash Y_1)=|\Gamma||\Lambda|\binom{n}{2}+|\Lambda|n+|\widehat{\Lambda}|+|Y_1|n+|C|+|\widehat{Y_0}|.\]

If $C=\emptyset$, then $M(A)\backslash Y_1=M(A)\backslash Y_1/C$. Keeping in mind that $n=r+r_C+t-|X|$, some arithmetic shows that this proves the result in the case where $C=\emptyset$. Thus, we now assume $C\neq\emptyset$.

One can see that $\varepsilon(M(A)\backslash Y_1)$ increases as $r_C$ increases, since $n=r+r_C+t-|X|$. Thus, to achieve maximum density, we should take $C$ to be independent, if possible. This can easily be achieved since $\Delta[C]=\Gamma(\mathbb{F}_p^C)$. Thus, we take $r_C$ to be equal to $|C|$ and $n=r+|C|+t-|X|$. In fact, let $A[B-X,C]$ be equal to
\begin{center}
\begin{tabular}{|c|}
\hline
 $I_{|C|}$\\
\hline
$I_{|C|}$\\
\hline
$I_{|C|}$\\
\hline
$1\cdots1$\\
$1\cdots1$\\
\hline
\multirow{2}{*}{0}\\
\\
\hline
\end{tabular}.
\end{center}

This implies that $n\geq 3|C|+2$ and, therefore, $r=n-|C|+|X|-t\geq2|C|+|X|-t+2$.

\begin{claim}
\label{cla:no-parallels}
 For every pair $\{e,f\}\subseteq E-(C\cup Y_1)$, the set $C\cup\{e,f\}$ is independent in $M(A)$.
\end{claim}

\begin{subproof}
Note that every column of $A[B-X,E-(C\cup Y_1)]$ has at most two nonzero entries. So the columns of $A[B-X,E-(C\cup Y_1)]$ labeled by $e$ and $f$ have nonzero entries in at most four rows. We proceed by induction on $|C|$. Suppose $|C|=1$. Since the single column of $A[B-X,C]$ has five nonzero entries, there is a unit row in $A[B-X,C\cup\{e,f\}]$ whose nonzero entry is in $C$. This implies that $r(C\cup\{e,f\})=r(\{e,f\})+1$. Since $e$ and $f$ are not parallel elements, $C\cup \{e,f\}$ must be independent.

Now suppose $|C|>1$. Then there are at least $3|C|\geq6$ unit rows in $A[B-X,C]$. Since $e$ and $f$ have nonzero entries in at most four rows, this implies that there is a unit row in $A[B-X,C\cup\{e,f\}]$ with its nonzero entry in a column labeled by some element $c\in C$. Thus $r(C\cup\{e,f\})=r((C-c)\cup\{e,f\})+1$. By the induction hypothesis, $(C-c)\cup\{e,f\}$ is independent. Therefore, $C\cup\{e,f\}$ is independent also.
\end{subproof}

Claim \ref{cla:no-parallels} implies that, when $C$ is contracted, the resulting matroid is still simple. Thus, $\varepsilon(M)=\varepsilon(M(A)\backslash Y_1)-|C|=$ \[|\Gamma||\Lambda|\binom{n}{2}+|\Lambda|n+|\widehat{\Lambda}|+|Y_1|n+|\widehat{Y_0}|.\] Some arithmetic, recalling that $n=r+|C|+t-|X|$, shows that this implies the result.
\end{proof}

\begin{lemma}
 \label{lem:extreamaltemplate}
Suppose Hypothesis \ref{hyp:cliquetemplate} holds, and let $\mathbb{F}$ be a finite field. Let $\mathcal{M}$ be a quadratically dense minor-closed class of $\mathbb{F}$-represented matroids, and let $\{\Phi_1,\dots,\Phi_s,\Psi_1,\dots,\Psi_t\}$ be the set of templates given by Hypothesis \ref{hyp:cliquetemplate}. For all sufficiently large $r$, the extremal matroids of $\mathcal{M}$ are the extremal matroids of the templates in some subset of $\{\Phi_1,\dots,\Phi_s\}$.
\end{lemma}

\begin{proof}
Let $p$ be the characteristic of $\mathbb{F}$. Since $\mathcal{M}$ is a quadratically dense minor-closed class and since $\varepsilon(\mathrm{PG}(r-1,\bFp))=\frac{p^r-1}{p-1}$, for all sufficiently large $r$, no member of $\mathcal{M}$ contains $\mathrm{PG}(r-1,\bFp)$ as a minor. By Hypothesis \ref{hyp:cliquetemplate}, there are a pair of integers $k,n$ such that every simple vertically $k$-connected member of $\mathcal{M}$ with an $M(K_n)$-minor is a member of at least one of the classes $\mathcal{M}(\Phi_1),\ldots,\mathcal{M}(\Phi_s)$.

By Lemmas \ref{lem:no_copies} and \ref{lem:polynomial}, for every frame template $\Phi$ and for all sufficiently large $r$, the size of a rank-$r$ extremal matroid of $\Phi$ is given by a quadratic polynomial in $r$. Thus, for all sufficiently large $r$, the size of the largest simple rank-$r$ matroid that conforms to some template in $\{\Phi_1,\dots,\Phi_s\}$ is given by a quadratic polynomial $h_{\mathcal{M}}'(r)$.

By definition, $h_{\mathcal{M}}(r)\geq h'_{\mathcal{M}}(r)$. We wish to show that equality holds for all sufficiently large $r$. Suppose otherwise. Then, for infinitely many $r$, we have $h_{\mathcal{M}}(r)>h'_{\mathcal{M}}(r)$. Theorem \ref{thm:geelen-nelson}, with $h_{\mathcal{M}}'(r)$ playing the role of $p(n)$ and with $k$ playing the role of $s$, implies that, for infinitely many $r$, there is a vertically $k$-connected rank-$r$ matroid $M_r\in\mathcal{M}$ with $\varepsilon(M_r)>h'_{\mathcal{M}}(r)$. Thus, these $M_r$ do not conform to any template in $\{\Phi_1,\dots,\Phi_s\}$. By Hypothesis \ref{hyp:cliquetemplate}, these $M_r$ contain no $M(K_n)$ minor. However, by Theorem \ref{thm:excludegraph}, there is an integer $c$ such that $\varepsilon(M_r)\leq cm$. This contradicts the fact that $\varepsilon(M_r)>h'_{\mathcal{M}}(r)$ for all $r$. By contradiction, we determine that $h_{\mathcal{M}}(r)=h'_{\mathcal{M}}(r)$, for all sufficiently large $r$.

Therefore, we know that, for all sufficiently large $r$, the extremal function $h_{\mathcal{M}}(r)$ is given by a quadratic polynomial. Now, Lemma \ref{lem:vertextremal} implies that, for all sufficiently large $r$, the rank-$r$ extremal matroids of $\mathcal{M}$ are vertically $k$-connected. Thus, by Hypothesis \ref{hyp:cliquetemplate}, it suffices to show that, for all sufficiently large $r$, the largest simple matroids of rank $r$ contain $M(K_n)$ as a minor. Suppose otherwise. Then, for infinitely many $r$, the largest simple matroids in $\mathcal{M}$ of rank $r$ have no $M(K_n)$-minor. By Theorem \ref{thm:excludegraph}, for infinitely many $r$, the largest simple matroids in $\mathcal{M}$ of rank $r$ have size at most $cr$, for some integer $c$. This contradicts the quadratic density of $\mathcal{M}$.
\end{proof}

Because of Lemma \ref{lem:extreamaltemplate}, the results in ~\cite{gvz17} regarding extremal functions (~\cite[Theorem 1.3]{gvz17}, ~\cite[Theorem 4.2]{gvz17}, and ~\cite[Corollary 4.10]{gvz17}) now require Hypothesis \ref{hyp:cliquetemplate}. However, ~\cite[Theorem 4.2]{gvz17} was also proved (independently of \cite{ggw15}) by Geelen and Nelson \cite[Corollary 1.6]{gn15}. Furthermore, ~\cite[Corollary 4.10]{gvz17} follows easily from ~\cite[Theorem 4.2]{gvz17}.

Finally, we modify \cite[Theorem 1.4]{gvz17} to require Tutte connectivity, rather than vertical connectivity.

\begin{theorem}
 \label{thm:1-flowing}
If Hypothesis \ref{hyp:connected-template} holds, then there exists $k\in\mathbb{Z}_+$ such that every simple, $k$-connected, 1-flowing matroid with at least $2k$ elements is either graphic or cographic.
\end{theorem}

Another paper that relies on the faulty version of the structure theorem is \cite{nvz15}. (It is stated as Hypothesis 4.1 in that paper.) The main result of \cite{nvz15} is \cite[Theorem 1.1]{nvz15}. In order to recover this result, \cite[Lemma 3.1]{nvz15} will need to be refined in order to deal with represented matroids close to the dual of a represented frame matroid. Since this case is significantly easier than the case in which the represented matroid is close to a represented frame matroid, Nelson and Van Zwam expect that the main result can be recovered, but elect to wait with their fix until the structure theorem has been proven by Geelen, Gerards, and Whittle.

\section{Refined Templates}
\label{sec:refined_templates}
In this section, we prove a result that will be of interest for future work. Nelson and Walsh gave Definition \ref{def:reduced} and proved Lemma \ref{lem:all-reduced} below.

\begin{definition}
 \label{def:reduced}
A frame template $\Phi=(\Gamma,C,X,Y_0,Y_1,A_1,\Delta,\Lambda)$ is \emph{reduced} if there is a partition $(X_0,X_1)$ of $X$ such that
\begin{itemize}
 \item $\Delta=\Gamma(\mathbb{F}^C_p\times\Delta')$ for some additive subgroup $\Delta'$ of $\mathbb{F}^{Y_0\cup Y_1}$,
\item$\mathbb{F}_p^{X_0}\subseteq\Lambda[X_0]$ while $\Lambda[X_1]=\{0\}$ and $A_1[X_1,C]=0$, and
\item the rows of $A_1[X_1,C\cup Y_0\cup Y_1]$ form a basis for a subspace whose additive subgroup is skew to $\Delta$.
\end{itemize}
\end{definition}

\begin{lemma}[{\cite[{Lemma 5.6}]{nw17}}]
 \label{lem:all-reduced}
Every frame template is equivalent to a reduced frame template.
\end{lemma}

Although \cite{nw17} was based on the faulty version of the structure theorem, Lemma \ref{lem:all-reduced} is a result about templates in their own right, rather than Conjecture \ref{con:template}, and is true for that reason. We will refer to the partition $X=X_0\cup X_1$ given in Definition \ref{def:reduced} as the \emph{reduction partition} of $\Phi$. We also introduce the following definition.

\begin{definition}
 \label{def:refined}
A frame template $\Phi=(\Gamma,C,X,Y_0,Y_1,A_1,\Delta,\Lambda)$ is \emph{refined} if it is reduced, with reduction partition $X=X_0\cup X_1$, and if $Y_1$ spans the matroid $M(A_1[X_1,Y_0\cup Y_1])$.
\end{definition}

\begin{remark}
Each of the minimal nontrivial binary frame templates given in \cite{gvz17} is reduced (often in trivial or vacuous ways) and, in fact, refined. 
\end{remark}

We wish to show that, for the purposes of using the Hypotheses given in Section \ref{sec:discussion}, only refined frame templates must be considered.

\begin{lemma}
\label{lem:notspanning}
 Let $\Phi=(\Gamma,C,X,Y_0,Y_1,A_1,\Delta,\Lambda)$ be a reduced frame template that is not refined. If $M\in\mathcal{M}(\Phi)$, then $E(M)-Y_0$ is not spanning in $M$.
\end{lemma}

\begin{proof}
Let $A$ be the matrix conforming to $\Phi$ such that $M=M(A)/C\backslash Y_1$. Since $\Phi$ is not refined, $Y_1$ does not span $M(A[Y_0\cup Y_1])$. Therefore, $Y_0$ contains a cocircuit in $M(A[Y_0\cup Y_1])$. In fact, since the definition of reduced implies that $A[X_1,E-(Y_0\cup Y_1\cup Z)]$ is the zero matrix, and since every column of $A[X_1,Z]$ is a copy of a column of $A[X_1,Y_1]$, we see that $Y_0$ contains a cocircuit in $M(A)$. This implies that $Y_0$ also contains a cocircuit in $M=M(A)/C\backslash Y_1$. Thus, $E(M)-Y_0$ is not spanning in $M$.
\end{proof}

Now we prove the main result of this section.

\begin{theorem}\label{thm:Y1spanning}
If Hypothesis \ref{hyp:connected-template} holds for a class $\mathcal{M}$, then the constant $k$, and the templates $\Phi_1,\ldots,\Phi_s,\Psi_1,\ldots,\Psi_t$ can be chosen so that the templates are refined. Moreover, If Hypothesis \ref{hyp:cliquetemplate} holds for a class $\mathcal{M}$, then the constants $k, n$, and the templates $\Phi_1,\ldots,\Phi_s,\Psi_1,\ldots,\Psi_t$ can be chosen so that the templates are refined.
\end{theorem}

\begin{proof}
Suppose that Hypothesis \ref{hyp:connected-template} holds, and let $\Phi\in\{\Phi_1,\ldots,\Phi_s$, $\Psi_1,\ldots,\Psi_t\}$. By Lemma \ref{lem:all-reduced}, we may assume that $\Phi$ is reduced with reduction partition $X=X_0\cup X_1$. Suppose for a contradiction that $\Phi$ is not refined. Choose $k\geq |Y_0|$. If $M$ is a $k$-connected represented matroid conforming to $\Phi$, then Lemma \ref{lem:notspanning} implies that $\lambda_M(Y_0)<r_M(Y_0)\leq |Y_0|$. Therefore, by $k$-connectivity, we must have $|E(M)-Y_0|<|Y_0|$. Thus, since $2k\geq 2|Y_0|$, we obtain a contradiction and conclude that the constant $k$, and the templates $\Phi_1,\ldots,\Phi_s$ can be chosen so that the templates are refined. Moreover, since $k$-connectivity is closed under duality, the templates $\Psi_1,\ldots,\Psi_t$ can be chosen to be refined as well.

Now suppose Hypothesis \ref{hyp:cliquetemplate} holds, and choose $k\geq|Y_0|$. Let $M$ be a simple, vertically $k$-connected member of some minor-closed class $\mathcal{M}$, and let $M$ have an $M(K_n)$-minor but no $PG(m-1,\bFp)$-minor for some positive integer $m$. Part (3) of Hypothesis \ref{hyp:cliquetemplate} implies that $M$ conforms to a template $\Phi\in\{\Phi_1,\ldots,\Phi_s\}$. By Lemma \ref{lem:all-reduced}, we may assume that $\Phi$ is reduced with reduction partition $X=X_0\cup X_1$. Suppose for a contradiction that $\Phi$ is not refined. By Lemma \ref{lem:notspanning}, $E(M)-Y_0$ is not spanning in $M$. This also implies that $\lambda_M(Y_0)<r_M(Y_0)\leq |Y_0|$. By vertical $k$-connectivity, $Y_0$ is spanning in $M$. Thus, $M$ is an $\mathbb{F}$-represented matroid of rank $r_M(Y_0)\leq|Y_0|$. Since $M$ is simple, we must have $|E(M)|\leq\frac{|\mathbb{F}|^{|Y_0|}-1}{|\mathbb{F}|-1}$; therefore, we set $n$ sufficiently large so that $|E(K_n)|=\binom{n}{2}>\frac{|\mathbb{F}|^{|Y_0|}-1}{|\mathbb{F}|-1}$. Thus, we see that the constants $k, n$, and the templates $\Phi_1,\ldots,\Phi_s$ can be chosen so that those templates are refined.

In the case where $M$ is a cosimple, cyclically $k$-connected member of $\mathcal{M}$ with an $M(K_n)$-minor but no $\mathrm{PG}(m-1,\bFp)$-minor for some positive integer $m$, we dualize the argument in the previous paragraph to conclude that the constants $k, n$, and the templates $\Psi_1,\ldots,\Psi_t$ can be chosen so that those templates are refined.
\end{proof}

\section*{Acknowledgements}
We would like to thank Jim Geelen, Peter Nelson, and Geoff Whittle for helpful discussions. In particular, Peter Nelson gave many helpful suggestions about the manuscript and alerted us to an error in one of the proofs.



\section*{Supplementary material}

# Appendix: Verifying Claim 3.6.8

In [3]:
```python
from itertools import combinations
from sage.matroids.advanced import *
from itertools import product
```

In [4]:
```python
# The complete ternary Dowling geometry:
def DowlingGeometry(r):
    A = Matrix(GF(3), r, r*r)
    i = 0
    for j in range(r):
        A[j,i] = 1
        i += 1
    for j in range(r-1):
        for k in range(j+1,r):
            A[j,i] = -1
            A[k,i] = 1
            i += 1
            A[j,i] = -1
            A[k,i] = -1
            i += 1
    return Matroid(A)
```

In [5]:
```python
def vname(path, ray, index):
    return path + str(ray) + '_' + str(index)
def ename(path, ray, index):
    return path + str(ray) + '_' + str(index)
```

In [6]:
```python
def perturbed_gadget_matroid(perturb_matrix, gadgets_per_ray,
                             num_rays=1):
    """
    Create the matrix representing "rays" of gadgets joined by graph
    edges. In the case where perturb_matrix=[0 0 0 0], num_rays = 4,
    and gadgets_per_ray=[5,2,2,2], the result is M(J'') from
    Figure 5.

    Edges are encoded <path><ray>_<position>. So element 'p2_1' is the
    second edge on path 'p' from the third ray (counts start at 0).
    The other paths are q and r.

    gadgets_per_ray is a vector if num_rays > 1

    The gadget elements are labeled <label><ray>_<position>. The
    labels are d,e,f,g. The ``root'' gadget is `d0_0, e0_0, f0_0,
    g0_0'. In Figure 5, the root gadget is labeled F4.
    """
    gadget_block = Matrix(GF(3), [[0,1,0,1],
```

```
                              [0,1,1,0],
                              [0,0,1,1],
                              [1,1,1,1]])   # matrix from Lemma 3.1
                                            # with the first six
                                            # columns removed
    if num_rays == 1:
        gadgets_per_ray = [gadgets_per_ray]
    num_rows = perturb_matrix.nrows() +
                sum(4 * g for g in gadgets_per_ray) -
                4 * (num_rays - 1) # Count root gadget only once.
    num_cols = sum(7 * (g - 1) for g in gadgets_per_ray) + 4
                # Each gadget other than the root gadget accounts for
                # seven elements -- the four elements of the gadget
                # and the three edges joining the gadget to the next
                # gadget. The root gadget gives four additional
                # elements.
    A = Matrix(GF(3), num_rows, num_cols)
    groundset = []
    vtx = {'p0_0': 0, 'q0_0': 1, 'r0_0': 2}   # Maps path vertices to
                                              # matrix rows

    imax = 3  # first unused row
    for ray in range(num_rays):
        vtx[vname('p',ray,0)] = vtx['p0_0']
        vtx[vname('q',ray,0)] = vtx['q0_0']
        vtx[vname('r',ray,0)] = vtx['r0_0']
        for i in range(1, gadgets_per_ray[ray]):
            vtx[vname('p',ray,i)] = imax
            vtx[vname('q',ray,i)] = imax + 1
            vtx[vname('r',ray,i)] = imax + 2
            imax += 3
    j = 0  # first unfilled column
    # Create the paths
    for ray in range(num_rays):
        for i in range(gadgets_per_ray[ray] - 1):
            # first path
            A[vtx[vname('p',ray,i)],j] = -1
            A[vtx[vname('p',ray,i+1)],j] = 1
            j += 1
            groundset.append(ename('p',ray,i))
            # second path
            A[vtx[vname('q',ray,i)],j] = -1
            A[vtx[vname('q',ray,i+1)],j] = 1
            j += 1
            groundset.append(ename('q',ray,i))
            # third path
            A[vtx[vname('r',ray,i)],j] = -1
            A[vtx[vname('r',ray,i+1)],j] = 1
            j += 1
            groundset.append(ename('r',ray,i))
    first_gadget = 0  # "root" gadget only gets added once
    for ray in range(num_rays):
        for k in range(first_gadget, gadgets_per_ray[ray]):
            first_gadget = 1 # "root" gadget only gets added once
```

```
                    # glue on the gadgets
                    elts = 'defg'
                    for l in range(4):
                        A[vtx[vname('p',ray,k)], j + l] = gadget_block[0, l]
                        A[vtx[vname('q',ray,k)], j + l] = gadget_block[1, l]
                        A[vtx[vname('r',ray,k)], j + l] = gadget_block[2, l]
                        A[imax, j + l] = gadget_block[3, l]
                    groundset.extend([ename(e,ray,k) for e in elts])
                    imax += 1
                    # put in the perturbation
                    A.set_block(num_rows - perturb_matrix.nrows(), j,
                            perturb_matrix)
                    j += 4
            return A, groundset
```

In [7]:
```
def check_perturbation(P):
        # Now we add a perturbation:
        A, gs = perturbed_gadget_matroid(P, [5,2,2,2], 4)
        M = Matroid(gs, A)
        contract_set = []
        # create negative loops on vertices of root gadget:
        contract_set.extend(['e1_1', 'f1_1', 'g1_1', 'e2_1', 'f2_1',
                             'g2_1', 'e3_1', 'f3_1', 'g3_1', 'p1_0',
                             'q2_0', 'r3_0'])
        # contract one element from each gadget on the main ray, except
        # the last:
        contract_set.extend(['d0_1', 'e0_2', 'f0_3', 'g0_4'])
        # contract the edges of the main ray:
        contract_set.extend([ename('p',0,i) for i in range(4)])
        contract_set.extend([ename('q',0,i) for i in range(4)])
        contract_set.extend([ename('r',0,i) for i in range(4)])
        N = (M / contract_set).simplify()
        return DowlingGeometry(N.rank()).has_minor(N) # We check if N is
                                                      # signed-graphic.
```

In [8]:
```
# We need to construct a list of all possible perturbations.
# We begin the list with the rank-0 perturbation.
perturbation_list=[Matrix(GF(3), [[0,0,0,0]])]
```

```
In [9]:  # Now we add all possible rank-1 perturbations.
         bases_already_checked=[]
         X=set([0,1,2,3])
         for B in Subsets(X,1):
             nonB=X.difference(B)
             for T in product(range(3), repeat=3): # There are 3 entries
                                                   # outside the identity
                                                   # matrix.

                 A=Matrix(GF(3), 1, 4)
                 i=0
                 for b in B:  # We make an identity matrix with columns
                              # labeled by the basis B.
                     A[i,b]=1
                     i=i+1
                 A[0,list(nonB)[0]]=T[0]
                 A[0,list(nonB)[1]]=T[1]
                 A[0,list(nonB)[2]]=T[2]
                 if all((A.matrix_from_columns(S)).det()==0 for S in
                     bases_already_checked):
                     # We only need to add a matrix to the list of
                     # perturbations if it contains no basis for which we
                     # already found all possible perturbations. This is
                     # because row operations allow us to transform the
                     # submatrix indexed by that basis into an identity matrix.
                     perturbation_list.append(A)
             bases_already_checked.append(B)
```

```
In [10]:  # Now we add all possible rank-2 perturbations.
          bases_already_checked=[]
          X=set([0,1,2,3])
          for B in Subsets(X,2):
              nonB=X.difference(B)
              for T in product(range(3), repeat=4): # There are 4 entries
                                                     # outside the identity
                                                     # matrix.

                  A=Matrix(GF(3), 2, 4)
                  i=0
                  for b in B: # We make an identity matrix with columns
                              #  labeled by the basis B.
                      A[i,b]=1
                      i=i+1
                  A[0,list(nonB)[0]]=T[0]
                  A[1,list(nonB)[0]]=T[1]
                  A[0,list(nonB)[1]]=T[2]
                  A[1,list(nonB)[1]]=T[3]
                  if all((A.matrix_from_columns(S)).det()==0 for S in
                          bases_already_checked):
                      # We only need to add a matrix to the list of
                      # perturbations if it contains no basis for which we
                      # already found all possible perturbations. This is
                      # because row operations allow us to transform the
                      # submatrix indexed by that basis into an identity matrix.
                      perturbation_list.append(A)
              bases_already_checked.append(B)
```

```
In [11]:  # Now we add all possible rank-3 perturbations.
          bases_already_checked=[]
          X=set([0,1,2,3])
          for B in Subsets(X,3):
              nonB=X.difference(B)
              for T in product(range(3), repeat=3): # There are 3 entries
                                                     # outside the identity
                                                     # matrix.

                  A=Matrix(GF(3), 3, 4)
                  i=0
                  for b in B: # We make an identity matrix with columns
                              # labeled by the basis B.
                      A[i,b]=1
                      i=i+1
                  A[0,list(nonB)[0]]=T[0]
                  A[1,list(nonB)[0]]=T[1]
                  A[2,list(nonB)[0]]=T[2]
                  if all((A.matrix_from_columns(S)).det()==0 for S in
                          bases_already_checked):
                      # We only need to add a matrix to the list of
                      # perturbations if it contains no basis for which we
                      # already found all possible perturbations. This is
                      # because row operations allow us to transform the
                      # submatrix indexed by that basis into an identity matrix.
                      perturbation_list.append(A)
          bases_already_checked.append(B)
```

```
In [12]:  # Now we add the only possible rank-4 perturbation.
          perturbation_list.append(matrix.identity(GF(3),4))
```

```
In [12]:  # We check all possible perturbations to see if any are
          # signed-graphic matroids.
          if any(check_perturbation(perturbation_list[t]) != False
                  for t in range(len(perturbation_list))):
              print "signed-graphic matroid"
          else:
              print "No perturbation is signed-graphic."
```

No perturbation is signed-graphic.

```
In [13]:  version()
```

Out[13]:  'SageMath version 8.0, Release Date: 2017-07-21'

*E-mail address*: `kgrace3@lsu.edu`

*E-mail address*: `svanzwam@math.lsu.edu`

Department of Mathematics, Louisiana State University, Baton Rouge, Louisiana


\end{document}